\documentclass{amsart}

\usepackage{stmaryrd}
\usepackage{enumerate}
\usepackage{enumitem}
\usepackage{multicol}
\usepackage{amssymb,amsmath}
\usepackage{amsthm}
\usepackage{turnstile}
\usepackage{textcomp}
\usepackage{graphicx}
\usepackage{subcaption}
\usepackage{mathtools}
\usepackage{wrapfig}
\usepackage{hhline}
\usepackage{array,multirow}
\usepackage{xcolor}
\usepackage{float}
\usepackage{geometry}
\usepackage{tikz}
\usetikzlibrary{positioning,calc}
\usepackage{algorithm}
\usepackage{algpseudocode}
\usepackage[hidelinks]{hyperref}
\usepackage{mathdots}
\usepackage{mathrsfs}

\algblock{Input}{EndInput}
\algnotext{EndInput}
\algblock{Output}{EndOutput}
\algnotext{EndOutput}

\theoremstyle{plain}
\newtheorem{theorem}{Theorem}[section]
\newtheorem{lemma}[theorem]{Lemma}
\newtheorem{corollary}[theorem]{Corollary}
\newtheorem{definition}[theorem]{Definition}
\newtheorem{example}[theorem]{Example}
\newtheorem{remark}[theorem]{Remark}
\newtheorem{notation}[theorem]{Notation}
\newtheorem{fact}[theorem]{Fact}
\newtheorem{proposition}[theorem]{Proposition}
\newtheorem*{claim}{Claim}

\DeclareMathOperator{\A}{\mathcal{A}}

\DeclareMathOperator{\B}{\mathcal{B}}
\DeclareMathOperator{\F}{\mathcal{F}}
\DeclareMathOperator{\N}{\mathbb{N}}

\newcommand{\arity}{\mathsf{arity}}
\newcommand{\is}[1]{\mathsf{is}_{#1}}
\newcommand{\rep}{\mathsf{rep}}

\begin{document}

\title{Almost free algebras: from the word problem to elimination of quantifiers}
\date{\today}

\author{Yifan Jia}
\email{kaiakirvan@gmail.com}
\address{University of Electronic Science and Technology of China, P.R.China}

\author{Heer Tern Koh}
\email{heertern001@e.ntu.edu.sg}
\address{University of Electronic Science and Technology of China, P.R.China}

\author{Bakh Khoussainov}
\email{bmk@uestc.edu.cn}
\address{University of Electronic Science and Technology of China, P.R.China}

\begin{abstract}
Term algebras are important objects in computer science and are correspondingly well-studied.
A natural generalization is to quotient these algebras by finitely many ground term equations, obtaining what we call almost free algebras.
One of the earliest results on almost free algebras is that their word problem is polynomial time decidable.
In this paper, we show that other natural problems: finding canonical representatives; computing the cardinality of a congruence class; checking if all congruence classes are infinite; checking if the algebra is finite; checking if two algebras are isomorphic, are all polynomial time decidable.
Another famous result regarding term algebras is that they admit quantifier elimination in a suitably expanded language.
Following this pattern, we also show that almost free algebras admit quantifier elimination by expanding the language with the standard tester predicates.
While this is implied by existing results, we view our main contribution here as providing a different approach, which we posit can be easily extended to a larger class that is not covered by existing works.
Finally, we provide an application to the quantifier elimination procedure, constructing examples of non-initial algebras over arbitrary signatures with a polynomial time word problem.
\end{abstract}

\maketitle

\section{Introduction}

Term algebras are important in computer science. They provide a framework for representing and manipulating syntactic structures, serving as the backbone for programming languages, type theory, and formal methods. By modeling abstract syntax trees (ASTs) as terms built from constructors, they enable techniques such as structural induction and primitive recursion for defining functions and proving program correctness. Term algebras also form the basis of initial algebra semantics, ensuring canonical representations for data types. They are essential in term rewriting systems, modeling computation via rewrite rules and underpinning equational reasoning, program optimization, and automated theorem proving. They support compiler construction through AST manipulation, unification in logic programming, and the design of domain-specific languages, bridging theory and practice in computer science.

Let $\mathcal F$ be the term algebra from constants of a signature $\Sigma$. The terms in $\mathcal F$ are called {\em ground terms}. Let us impose a finite set of ground term equations $\Gamma$ on $\mathcal F$. Each equation $p=q$ in $\Gamma$ asserts that the terms $p$ and $q$ denote the same abstract value. The quotient $\mathcal F_\Gamma$ is obtained by identifying terms that can be transformed into each other using $\Gamma$. This forms the algebra of terms modulo the congruence generated by $\Gamma$. If $\Gamma$ is empty, then $\mathcal F_\Gamma$ is $\mathcal F$ itself. Adding equations collapses the structure, identifying some terms while leaving others separate, and yields an infinite algebra (unless the equations force finiteness) that retains a tree-like flavor but now embodies equational constraints. This construction is entirely natural: it mirrors how we define data types by generators and relations in algebra, how we specify abstract data types in programming, and how we model computational effects in term rewriting systems. We call thus obtained algebras almost free algebras. 

The almost free algebra $\mathcal F_\Gamma$ sits at the crossroads of several fundamental concerns in computer science and logic. It generalizes the well-understood free term algebra to a setting where finite equational information is built in, raising the question of whether decidability results true for term algebras survive. It provides semantics for abstract data types with ground equations and underpins ground term rewriting systems, where confluence and termination yield algorithmic insights. Investigating such term-like algebras offers a tool for automating reasoning about infinite structures generated by finite equational constraints, placing it at the intersection of universal algebra, logic, and computation. 

In this paper we investigate algorithmic and model-theoretic properties of almost free algebras. For instance, we provide polynomial time algorithms for: the canonical representative problem; the congruence class cardinality problem; the intrinsic infinity problem; the finiteness problem; and the isomorphism problem (each to be defined shortly).
We also provide a new quantifier elimination process for almost free algebras thus generalising the well-known classical result that the term algebra possesses quantifier elimination.

\subsection{Background}
Quantifier elimination of term algebras goes back to Mal'cev \cite{mal1936}.
Rabin later reduced the decidability problem for term algebras to SnS in \cite{Rabin1969} by representing terms as trees.
Another quantifier elimination procedure for absolutely free algebras in an expanded finite language was introduced by Belegradek \cite{Belegradek1988}, and subsequently refined by Hodges \cite{Hodges1993}. 
Sturm and Weispfenning proposed in \cite{sturm2002} to construct sample solutions for the existential formulas and showed that the resulting decision procedure lies in the fourth Grzegorczyk complexity class. This is consistent with the result of Compton and Henson \cite{compton1990}, who proved that no quantifier-elimination procedure for such algebras can be elementary recursive. Mal'cev also extended his result to the classes of locally free algebras with symmetry conditions.

Although the term algebras have a decidable first-order theory, introducing internal relations can 
complicate their logical structure. Tulipani analyzed
term algebras enriched with the subterm relation with Marongiu and showed in \cite{Marongiu1993} that such an extension leads to $\Sigma_1^1$-completeness of the $\exists\Delta$ fragment.
He also identified fragments in \cite{Tulipani1994} that remain decidable, proving in particular that the existential fragment of the theory is decidable. 

External extensions of $\mathcal F$  have also been investigated. Manna, Sipma, and Zhang investigated the term algebras equipped with a length function and Presburger arithmetic constraints, whose complexity is $2k$-fold exponential for $k$ quantifier alternations \cite{Zhang2004}; this was later improved to k-fold exponential in \cite{Zhang2006}.  Voronkov and Korovin showed in \cite{korovin2000} the decidability of the existential fragment of term algebras equipped with an ordering induced by assigning weights to terms via an arbitrary weight function on the signature, commonly known as the Knuth–Bendix ordering. Voronkov and Rybina further established in \cite{Rybina2001} the decidability of the theory of term algebras extended with queues, which constitute the only class of abstract data objects that cannot be represented as many-sorted term algebras. 

In contrast, quotients of free term algebras are not as well-studied. 
D. Kozen studied quotients of $\mathcal F$ obtained by factoring out a finite set of equations. He proved that  the word problem for such algebras is in $P$ by \cite{kozen77}. Comon presented a set of rules for quantifier elimination in term algebras modulo quasi-free congruences in \cite{comon1993}, which generalises Mal'cev’s result.
Khoussainov and Rubin introduced a more general setting in \cite{Bakhadyr2005}, namely the algebra freely generated by a partial algebra $\A$. They showed that if $\A$ admits quantifier elimination, then so does its free total extension $\F(\A)$. However, the proofs in \cite{Bakhadyr2005} are sketchy and hence unclear, with many details left to the reader.

\subsection{Our Contribution}\label{sec:contribution}
The rest of the paper is organized as follows.
In Section \ref{sec:almostfree}, we provide some basic preliminaries and introduce the main object of interest, almost free algebras, characterizing these algebras as exactly those which are the free extensions of finite partial algebras.

In Section \ref{sec:algprop}, we study the following natural questions regarding almost free algebras.
In the following, fix $\Gamma$.
(1) {\bf The canonical representative problem}: Given a term $t$, compute a canonical representative in the $\sim_{\Gamma}$-congruence class of $t$---$\rep(t)$ such that if $s=t$ (in the almost free algebra), then $\rep(s)=\rep(t)$ (as terms); (2) {\bf The congruence class cardinality problem}: Given a term $t$, compute the cardinality of the $\sim_{\Gamma}$-congruence class of $t$ if it is finite, and return $\infty$ otherwise; (3) {\bf The intrinsic infinity problem}: Is every $\sim_{\Gamma}$-congruence class infinite; (4) {\bf The finiteness problem}: Is the almost free algebra $\F_{\Gamma}$ finite; (5) {\bf The isomorphism problem}: Given $\Sigma,\Gamma_{1},\Gamma_{2}$, check if $\F_{\Gamma_{1}}\cong\F_{\Gamma_{2}}$.
We show that each of these is polynomial time solvable.

In Section \ref{sec:qe}, we provide the quantifier elimination process for almost free algebras, showing that its theory is decidable.
We adapt the approach by Mal'cev \cite{mal1936}, also utilising ideas from Khoussainov and Rubin \cite{Bakhadyr2005}.
We point out that while our proof is the same in spirit as Khoussainov's and Rubin's, there is a non-trivial difference in that we do not place as many `constraints' upon the formulas, in particular the free variables, which is the part of their proof we found unconvincing.
We remark also that our result is implied by Comon's more general result, but that we take a different more direct approach which may admit an easy extension that is not implied by Comon's result (see Remark \ref{rem:extension}).
Roughly speaking, Comon's approach requires that the equations mentions only terms of height $\leq 1$, and then showing that every set of ground term equations can be rewritten to satisfy this property by introducing new constant symbols to the algebra.
In this paper, we perform the quantifier elimination procedure only expanding our language with tester predicates.

Finally, in Section \ref{sec:app}, we provide applications of the quantifier elimination procedure, constructing a non-initial algebra (i.e., one that cannot be axiomatised by only finitely many equations), but yet has a polynomial time word problem.
This can be thought of as contrasting Kozen's result that every almost free algebra (which is initial) has polynomial time word problem, and the classical result that there are groups (which are also initial algebras\footnote{The group axioms can be formulated as a finite set of equations involving variables, and any further relation between the generators can similarly be expressed as ground term equations.}) with undecidable word problems \cite{nov55,boone58}.

\section{Almost free algebras}\label{sec:almostfree}

Let $\Sigma$ be a finite functional signature having  function  symbols $f,  g, \ldots$ and constant symbols $c_0, c_1, \ldots$. 
Structures of $\Sigma$ are {\bf algebras} that we denote by $\mathcal A=(A; f^{\mathcal A}, g ^{\mathcal A}, \ldots,  c_0^{\mathcal A}, c_1^{\mathcal A}, \ldots)$.
If there is no confusion, we often omit the superscripts $\mathcal A$ from $f^{\mathcal A}$.    An algebra generated by the constants of $\Sigma$ is called {\bf $\bar{c}$-generated}. Every $c$-generated algebra is finitely generated.  By definition,  all algebras $\mathcal A$  are such that for all $f\in \Sigma$, their interpretations $f^{\mathcal A}$ are total operations. We, however, allow {\bf partial algebras}, where some interpretations $f^{\mathcal A}$ of $n$-ary function symbols might be partial operations, that is, on some tuples $(a_1, \ldots, a_n)$ the values $f(a_1, \ldots, a_n)$ are undefined.

\medskip

We now define {\bf terms} of $\Sigma$.  
All variable and constant symbols are  terms. If $t_1$, $\ldots$, $t_n$ are terms and $f\in \Sigma$ is an $n$-ary function symbol, then the expression $f(t_1, \ldots, t_n)$  is a  term.
We can turn the set of terms into the {\bf term algebra} as follows.
The domain of the algebra is the set of all terms. The interpretation of
each $f\in \Sigma$ is given by the following rule. If $f$ is a constant $c$, then the interpretation of $f$ is $c$ itself. 
If $f$ has arity $n>0$, then the interpretation of $f$ is such that the value of $f$ on tuple of terms $(t_1, \ldots, t_n)$ is the term $f(t_1, \ldots, t_n)$.  

\begin{notation}
From now on, we sometimes denote terms in Polish notation (without parentheses), for instance, $ft_{1}\dots t_{k}$, and reserve parentheses, say $f(t_{1},\dots,t_{k})$ for when we mean to evaluate the function symbol on the given inputs.
\end{notation}

A term $t$ is a {\bf ground term} if $t$ has no variables. The  {\bf tree representations}  $\Omega(t)$ of the ground term $t$ is defined as below:

\begin{definition}\label{def:tree}
For a ground term $t$, define the {\bf tree representation of $t$}, denoted $\Omega(t)$, to be a labeled tree $\Omega(t)\subseteq\Sigma\times\N^{<\N}$ as follows:
\begin{itemize}
\item For each constant $b$, $\Omega(b)=\{(b,\epsilon)\}$.
\item Let $t=ft_{0}t_{1}\dots t_{\arity(f)-1}$. Then define
$$
\Omega(t)=\{(f,\epsilon)\}\cup\bigcup_{i<\arity(f)}\{(\theta,i^{\frown}\sigma)\mid (\theta,\sigma)\in\Omega(t_{i})\}.
$$
\end{itemize}
We abuse notation and write $\Omega^{-1}(t,\sigma)$ for  the subterm of $t$ rooted at $\sigma$. The {\bf height of $t$}, denoted by $\Delta(t)$, is the height of $\Omega(t)$.
The {\bf size of $t$}, denoted $|t|$, is $|\Omega(t)|$.
We  might identify $t$ with $\Omega(t)$ without mentioning it explicitly. 
\end{definition}

Let $F$ be the set of all {\bf ground terms}, terms generated by only constant and function symbols.
The set of all ground terms is a subalgebra of the algebra of terms. We call this subalgebra the {\bf ground term algebra} and denote it by $\mathcal F$.  The ground term algebra $\F$ is clearly $\bar c$-generated. Moreover, the algebra $\mathcal F$ is {\bf universal} in the sense that every $\bar{c}$-generated algebra is a homomorphic image of the ground term algebra $\mathcal F$. Any universal $\bar c$-generated algebra is isomorphic to $\mathcal F$.

\begin{definition}
 Key concepts of this paper are the following:
\begin{enumerate}
\item An {\bf equation} is an expression of the form $p=q$ where $p$ and $q$ are terms.  
\item If $p$ and $q$ are ground terms, then we call the equation  $p=q$
a {\bf ground term equation}. 
\item     A {\bf finite presentation} is a finite set $\Gamma$ of term equations. 
\end{enumerate}
\end{definition}

For a finite presentation $\Gamma$, we define the relation $\sim_{\Gamma}$ on $\mathcal F$:
$p\sim_{\Gamma} q $ if $\Gamma$ proves $p=q$, that is, $\Gamma \vdash p=q$. The relation $\sim_{\Gamma}$ is a congruence relation on $\mathcal F$. 
Hence, we have the quotient algebra $\mathcal F/\sim_\Gamma$. 
An algebra $\mathcal A$ is {\bf finitely presented} if there is exists a finite set $\Gamma$ of equations such that $\mathcal A$ is isomorphic to $\mathcal F/\sim_\Gamma$. For instance, the two generated free group $F_2$ is a finitely presented algebra over the signature
$f,g, a,b, e$, where $f$ is binary, $g$ is unary, and $a,b,e$ are constants. A finite presentation of $F_2$ is then this:
$$
\Gamma=\{fxfyz=ffxyz, \  ggx=x, \ fxgx=e, \  fxe=x, \ fex=x\}. 
$$
The group $F_2$ is then isomorphic to $\mathcal F/\sim_\Gamma$. In this paper, we are interested in those presentations $\Gamma$ that have no variables. We now single out these algebras:

\begin{definition}\label{def:almostfree}
An algebra $\A$ is {\bf almost free} if there exists a finite set $\Gamma$ of ground term equations such that $\mathcal A$ is isomorphic to $\mathcal F/ \sim_{\Gamma}$. Denote the quotient $\mathcal F/ \sim_{\Gamma}$ by $\mathcal F_{\Gamma}$
\end{definition}

Assume that $\Gamma$ is a finite set of ground term equations. Then the congruence relation $\sim_{\Gamma}$ can be described as follows. 
We write $p\rightarrow_{\Gamma} q$ if there is an equation $p'=q'\in \Gamma$ or $q'=p'\in\Gamma$ such that $p$ contains a subterm $p'$ and $q$ is obtained from $p$ by replacing $p'$ with $q'$. If we represent $p$ as the tree $\Omega(p)$, then $p\rightarrow q$ corresponds to replacing the subtree $\Omega(p')$ of the tree $\Omega(p)$  with the tree $\Omega(q')$. By $\rightarrow^\star_{\Gamma}$ we denote the transitive closure of the relation $\rightarrow_{\Gamma}$. The following is a standard result:

\begin{proposition}
Let $\Gamma$ be a finite set of ground term equations. Then for all $p,q \in \mathcal F$, the following three conditions are equivalent: \ (1) $\Gamma \vdash p=q$. (2)  $p\rightarrow^\star_{\Gamma} q$.  (3) $\mathcal F_{\Gamma} \models p=q$. \qed 
\end{proposition}

{\bf Partial algebras}  are structures of $\Sigma$ where the function symbols $f \in \Sigma$ (of arity $n$) are interpreted as partial operations.
However, we still postulate that these are generated by values of constants of $\Sigma$.  
\begin{definition}
We say that an algebra $\mathcal A$ is {\bf free over $\mathcal B$} if (1) $\mathcal B$ is a substructure of $\mathcal A$, (2) every algebra $\mathcal C$ that contains $\mathcal B$ as a substructure is a homomorphic image of $\mathcal A$. 
\end{definition}
We note that  any two free algebras over $\mathcal B$ are isomorphic. 

Let $\mathcal F(\mathcal B )$ be the free algebra over $\mathcal B$. Intuitively,  $\F(\B)$ is the set of terms obtained by applying the function symbols to elements of $\B$, but evaluating each subterm and replacing it with an element from $\B$ whenever possible. Here is a constructive definition of  $\F(\B)$:
\begin{enumerate}
    \item All elements of $\B$ are called {\bf $\B$-terms}.
    \item Let $f\in \Sigma$ be an $s$-ary function symbol. For all $\B$-terms $t_1$, $\ldots$, $t_s$, the expression $ft_1\ldots t_s$, where $f$ is an $s$-ary function symbol from $\Sigma$, is a {\bf $\B$-term} if and only if $f(t_1, \ldots, t_s)^{\B}$ is undefined in $\B$.
\end{enumerate}
The domain of $\F(\B)$ is the set of all $\B$-terms. (We will often use the same symbol to denote both the algebra and its domain.)
An $s$-ary function symbol $f\in\Sigma$ is interpreted as follows. If all $t_1, \ldots, t_s$ belong to $\B$ and $f(t_1,\ldots, t_s)^{\B}=b\in\B$, then $f(t_1, \ldots, t_s)^{\F(\B)}=b$. Else, the value of $f(t_1, \ldots, t_s)^{\F(\B)}$ is the expression $ft_1\ldots t_s$ itself.

As a simple example, consider the partial algebra with domain $\mathcal C$, the set of all constants of $\Sigma$, such that for all function symbols $f\in \Sigma$, the interpretation of $f$ on $\mathcal C$ is undefined on all tuples of $\mathcal C$. This defines the partial algebra $\mathcal C$. Then the free  algebra over $\mathcal C$, that is the algebra $\mathcal F (\mathcal C)$,  coincides with the ground term algebra $\mathcal F$. We now characterize almost free algebras as algebras free over finite partial algebras.

\begin{theorem}\label{thm:char}
An algebra $\A$ is almost free if and only if it is free over a finite partial algebra $\B$.
\end{theorem}

We first state two facts that will be used in the proof of Theorem \ref{thm:char} and \ref{thm:iso}.
\begin{fact}\rm
Fix a signature $\Sigma$ and consider a finite partial algebra $\B$ over $\Sigma$.
For any algebra $\A$ which contains $\B$ as a substructure, i.e., there is a homomorphic embedding of $\B$ into $\A$ (which fixes the constants), there exists a unique surjective homomorphism from $\F(\B)$ onto $\A$.
In categorical terms, $\F(\B)$ is the initial object in the category with objects algebras that contain $\B$ as a substructure and arrows given by surjective homomorphisms.
\end{fact}

\begin{fact}\rm
Fix a signature $\Sigma$ and a finite set of equations $\Gamma$.
In the category of algebras $\A$ such that $\A\vDash\Gamma$ with arrows surjective homomorphisms, $\F_{\Gamma}$ is the initial object.
\end{fact}

\begin{proof}[Proof of Theorem \ref{thm:char}]
Let $\A\cong\F_{\Gamma}$ where $\Gamma=\{a_{0}=b_{0},a_{1}=b_{1},\dots,a_{n}=b_{n}\}$ be almost free.
Consider $\B$ a partial algebra defined as follows.
Let $N=\max\{\Delta(a_{i}),\Delta(b_{i})\mid i\leq n\}$.
Then for each $[b]\in \A$, if there exists a term $b^{*}\in[b]$ so that $\Delta(b^{*})\leq N$, then let $[b]\in\B$.
Note that $\B$ is finite.
For such a $\B$, observe that any algebra $\mathcal{C}$ containing $\B$ as a substructure is a model of $\{a_{0}=b_{0},\dots,a_{n}=b_{n}\}$.
Thus, there is a unique surjective homomorphism from $\A$ to $\mathcal{C}$.
Additionally, by the construction of $\B$, $\A$ also contains $\B$ as a substructure.
That is, $\A$ is the initial object in the category of algebras containing $\B$ as a substructure and hence isomorphic to $\F(\B)$.

For the converse, let $\B=(B;f_{0},f_{1},\dots, c_0, c_1, \ldots)$ be a finite partial algebra and let $\Gamma=\{f\mathbf{b}=c \mid f\in\Sigma,\,\mathbf{b}\in\B^{<\N},\,c\in \B,\,\B \models f\mathbf{b} =c\}$.
Just as before, any algebra $\mathcal{C}$ that contains $\B$ as a substructure is such that $\mathcal{C}\vDash\Gamma$.
Then, there is a unique surjective homomorphism from $\F_{\Gamma}$ to $\mathcal{C}$.
Since $\F_{\Gamma}$ also contains $\B$ as a substructure, then $\F_{\Gamma},\F(\B)$ are both initial objects in the category of algebras containing $\B$ as a substructure, and hence are isomorphic.
\end{proof}

\section{Algorithmic properties of almost free algebras}\label{sec:algprop}

\subsection{The equality problem (Kozen's algorithm)}\label{sec:kozen}
For completeness, we provide Kozen's algorithm for the equality problem here.
(It will be repeatedly used in the remainder of the paper.)
Given $\Gamma$ a finite set of ground term equations, consider the directed graph $R_{\Gamma}$ obtained by taking the disjoint union of $\Omega(p)$ where either $p=q\in\Gamma$ or $q=p\in\Gamma$ for some $q$, and replacing the label of each node $(\theta,\sigma)\in\Omega(p)$ with $(p,\theta,\sigma)$.
It is evident that $R_{\Gamma}$ can be produced in polynomial time in $|\Gamma|:=\sum_{p=q\in\Gamma}|p|+|q|$.

\begin{example}
Consider $\Gamma=\{a=fbc,c=fab\}$. Then $R_{\Gamma}$ will be the graph:
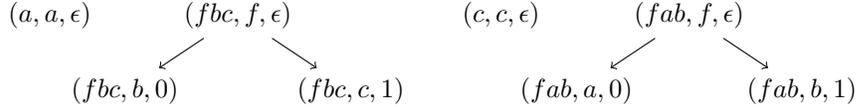
\begin{center}
\begin{tikzpicture}
\node (a) at (1,0) {$(a,a,\epsilon)$};
\node (fbc) at (3.5,0) {$(fbc,f,\epsilon)$};
\node (fbc0) at (2,-1) {$(fbc,b,0)$};
\node (fbc1) at (5,-1) {$(fbc,c,1)$};
\node (c) at (7,0) {$(c,c,\epsilon)$};
\node (fab) at (9.5,0) {$(fab,f,\epsilon)$};
\node (fab0) at (8,-1) {$(fab,a,0)$};
\node (fab1) at (11,-1) {$(fab,b,1)$};

\draw[->] (fbc) to (fbc0);
\draw[->] (fbc) to (fbc1);
\draw[->] (fab) to (fab0);
\draw[->] (fab) to (fab1);
\end{tikzpicture}
\end{center}
\end{example}

To determine whether $\mathcal F_{\Gamma}\models s=t$, we extend the graph $R_{\Gamma}$ by adjoining the trees $\Omega(s)$ and $\Omega(t)$, replacing each label $(\theta,\sigma)\in\Omega(s)$ (resp.~$\Omega(t)$) by $(s,\theta,\sigma)$ (resp.~$(t,\theta,\sigma)$).
Denote the resultant graph as $R_{\Gamma}(s,t)$.
It is easy to see that $R_{\Gamma}(s,t)$ can be produced in polynomial time in $|\Gamma|,|s|,|t|$.
The idea now is to `close' $R_{\Gamma}(s,t)$ under the axiom ``$\F_{\Gamma}\vDash fu_{1}\dots u_{\arity(f)}=fv_{1}\dots v_{\arity(f)}$ if $\F_{\Gamma}\vDash\bigwedge_{i<\arity(f)}u_{i}=v_{i}$'', and the transitivity of ``$=$'':
\begin{description}
\item[Step $0$] For each pair of nodes $(p,\theta,\sigma)$ and $(q,\xi,\tau)$, if $\Omega^{-1}(p,\sigma)=\Omega^{-1}(q,\tau)$ (as terms), or if one of $\Omega^{-1}(p,\sigma)=\Omega^{-1}(q,\tau)$ or $\Omega^{-1}(q,\tau)=\Omega^{-1}(p,\sigma)$ is contained in $\Gamma$, then add an undirected edge between the two nodes.

\item[Step $n$] For nodes $u,v,w$, if there is an undirected edge between $u,v$ and also between $v,w$, but no undirected edge between $u,w$, then add an undirected edge between $u,w$.

If nodes $(p,\theta,\sigma),(q,\xi,\tau)$ are such that $\theta=\xi=f$ and for each $l<\arity(f)$, there are undirected edges between $(p,\theta',\sigma^{\frown}l)$ and $(q,\xi',\tau^{\frown}l)$, but no undirected edge between $(p,\theta,\sigma)$ and $(q,\xi,\tau)$, then add one between them.

Terminate the algorithm at the stage where no new undirected edges are added.
\end{description}

Since the total number of possible undirected edges is quadratic in the total number of vertices, the algorithm above terminates in polynomial time in the size of $R_{\Gamma}(s,t)$ which is itself polynomial in $|\Gamma|,|s|,|t|$.
It is also not hard to see that if the algorithm ends with there being an edge between the nodes $(s,\theta,\epsilon)$ and $(t,\xi,\epsilon)$, then $\F_{\Gamma}\vDash s=t$.
We refer the reader to \cite{kozen77} for the converse.

\subsection{The canonical representative problem}\label{sec:min}
Recall that this is the problem of finding for each $\sim_{\Gamma}$-congruence class of $t$, a canonical representative $\rep(t)$ so that for any $s\sim_{\Gamma}t$, $\rep(s)=\rep(t)$ (as terms).
We now provide some key definitions that will be used in tackling the problems listed out in Section \ref{sec:contribution}.

\begin{definition}\label{def:type}
Let $\Gamma$ be a finite set of ground term equations.
Consider the graph $G_{\Gamma}=(V_{\Gamma},E_{\Gamma})$ where $V_{\Gamma}=\{p \mid  \exists q(p=q\in \Gamma \vee  q=p\in \Gamma)\}$, and $E_{\Gamma}=\{\{p,q\}\mid\F_{\Gamma}\vDash p=q\}$.
Then let $C_{1},\dots,C_{k}$ be the distinct connected components of $G_{\Gamma}$.
Then we say that a term $p\in V_{\Gamma}$ has {\bf type} $i$ if it is contained in $C_{i}$.
\end{definition}

By Kozen's algorithm, the process of obtaining the types of each term mentioned in $\Gamma$ (i.e., those terms $p\in V_{\Gamma}$) takes polynomial time in $|\Gamma|$.
In fact, one may extend the above definition to say that a term has {\bf type} $i$ if it is $\F_{\Gamma}$-equivalent to a term $p\in C_{i}$.
This is again polynomial time decidable in $|\Gamma|,|t|$ by Kozen's algorithm.

\begin{definition}\label{def:typedrep}
Let $\Gamma$ be a finite set of ground term equations which induces the types $1,\dots,k$.
Define the {\bf reduced typed representation} of $t$, denoted $r(t)\subseteq\big(\Sigma\cup\{1,\dots,k\}\big)\times\N^{<\N}$, as follows: 
Let $(\xi,\epsilon)\in r(t)$ where $\epsilon$ is the empty string and $(\xi,\epsilon)\in\Omega(t)$ (recall Definition \ref{def:tree}).
For $(\theta,\sigma)\in\Omega(t)$, let $(\theta,\sigma)\in r(t)$ if for all nonempty $\tau\preceq\sigma$, $\Omega^{-1}(t,\tau)$ has no type.
Otherwise, if $\Omega^{-1}(t,\sigma)$ has type $i$, and for each nonempty $\tau\prec\sigma$, $\Omega^{-1}(t,\tau)$ has no type, then let $(i,\sigma)\in r(t)$.
\end{definition}

Given Kozen's algorithm, it is easy to see that producing $r(t)$ takes polynomial time in $|t|$ and $|\Gamma|$.
We will also often refer to the nodes $\sigma$, or $v$, of the various representations, $\Omega(t),r(t)$ as the (sub)terms themselves.

\begin{theorem}\label{thm:cr}
There is a function $\rep$ computable in polynomial time that,  given $\Gamma$ a finite set of ground term equations and ground terms $p$ and $q$, outputs the terms $\rep(p)$ and $\rep(q)$ such that (1) $p\sim_{\Gamma} \rep(p)$, and (2) $p\sim_{\Gamma} q$ if and only if $\rep(p)=\rep(q)$.
\end{theorem}

\begin{proof}
Consider the following algorithm that produces $\Omega(\rep(t))$, a tree representation for the canonical term of the $\sim_{\Gamma}$-congruence class of $t$:
\begin{enumerate}
\item Begin with $T$ as the empty tree, and fix terms $p_{1},p_{2},\dots,p_{k}$ of type $1,2,\dots,k$ respectively.
\item Check if $t$ has type $i$.
If so, then return $\Omega(p_{i})$.
Otherwise, compute $r(t)$ and let $T=r(t)$.
\item Now, for each leaf $(\theta,\sigma)$ of $T$, do the following.
If $\theta=i$ for some $i>0$, then let
$$
T=\big(T\setminus\{(i,\sigma)\}\big)\cup\{(\xi,\sigma^{\frown}\tau)\mid(\xi,\tau)\in\Omega(p_{i})\}.
$$
Otherwise do nothing for this leaf and proceed to the next.
\end{enumerate}
It is easy to see that $T$ is a valid tree representation of a term, and thus, we may easily extract $\rep(t)$ from $T$ as desired.
The algorithm above clearly runs in polynomial time in $|\Gamma|,|t|$.
Intuitively, given $t$, the algorithm searches for the maximal nodes which are equivalent to some typed term and replaces them with the tree representation of a representative term $p_{i}$ of type $i$ chosen by us.
We show that the algorithm is correct by way of the next lemma.
\end{proof}

\begin{lemma}\label{lem:canonicalrep}
If $\F_{\Gamma}\vDash s=t$, then $\rep(s)=\rep(t)$.
\end{lemma}

\begin{proof}
We shall instead show the equivalent statement that if $\F_{\Gamma}\vDash s=t$, then $\Omega\big(\rep(s)\big)=\Omega\big(\rep(t)\big)$.
We first claim that for any ground term $t$, if $t=ft_{0}\dots t_{\arity(f)-1}$, and is such that $t$ has no type, then for each $i<\arity(f)$,
$$
\Omega\big(\rep(t_{i})\big)=\left\{(\theta,\sigma)\mid (\theta,i^{\frown}\sigma)\in\Omega\big(\rep(t)\big)\right\}.
$$
To see this, observe that to define $\rep(t)$, the algorithm picks out the maximal subterms of $t$ which are typed, and subsequently replacing the corresponding subterms with some chosen typed term.
By the assumption that $t$ has no type, it follows that any maximal subterm of $t$ which has a type is necessarily a maximal subterm of one of the $t_{i}$ which has a type and vice versa.
This proves the claim.

Now we proceed by induction on $\max\{\Delta(s),\Delta(t)\}$ to prove the lemma.
For the base case, we necessarily have that both $s,t$ are constants.
Then, if $\F_{\Gamma}\vDash s=t$, it must be that either $s,t$ are both already the same constant, or there is some sequence of replacements via $\Gamma$ which takes $s$ to $t$.
Then, we either have that $\rep(s)=s=t=\rep(t)$ (actually ``$=$'' and not just in $\F_{\Gamma}$), or $s,t$ necessarily have the same type, and so $\Omega\big(\rep(s)\big)=\Omega\big(\rep(t)\big)$.

Now suppose that $\max\{\Delta(s),\Delta(t)\}=n>0$ and that the statement holds for all pairs $s',t'$ with $\max\{\Delta(s'),\Delta(t')\}<n$.
Then either both $s,t$ have the same type, or neither are typed.
In the former case, we obtain that $\Omega\big(\rep(s)\big)=\Omega\big(\rep(t)\big)$ as desired.
In the latter, by the assumption that $\F_{\Gamma}\vDash s=t$, we must have that $s=fs_{0}s_{1}\dots s_{\arity(f)-1}$ and $t=ft_{0}t_{1}\dots t_{\arity(f)-1}$ and $\F_{\Gamma}\vDash\bigwedge_{i<\arity(f)}s_{i}=t_{i}$ as $s,t$ themselves have no type.
By induction, for each $i<\arity(f),\,\Omega\big(\rep(s_{i})\big)=\Omega\big(\rep(t_{i})\big)$.
Then applying our claim above, since neither of $s,t$ are typed, we have that
\begin{align*}
\Omega\big(\rep(s)\big)&=\{(f,\epsilon)\}\cup\left\{(\theta,i^{\frown}\sigma)\mid(\theta,\sigma)\in\Omega\big(\rep(s_{i})\big)\right\}\\
&=\{(f,\epsilon)\}\cup\left\{(\theta,i^{\frown}\sigma)\mid(\theta,\sigma)\in\Omega\big(\rep(t_{i})\big)\right\}\\
&=\Omega\big(\rep(t)\big)
\end{align*}
as desired.
\end{proof}

\subsection{The congruence class cardinality problem}\label{sec:size}
Recall that this is the problem of computing the size of the $\sim_{\Gamma}$-congruence class of a given term $t$.
We now define $R_{\Gamma}$ as in Section \ref{sec:kozen}, but with the additional data of types included in the nodes.
More formally, each label $(p,\theta,\sigma)$ is replaced by $(p,\theta,\sigma,i)$ if $\Omega^{-1}(p,\sigma)$ has type $i$, and $(p,\theta,\sigma,0)$ otherwise.
Additionally, for each $i>0$, also add undirected edges between each pair of distinct nodes with type $i$.
Let $R_{\Gamma}$ now denote this mixed graph.

\begin{example}
Consider $\Gamma=\{a=fbc,c=fab\}$. Then $R_{\Gamma}$ will be the graph:
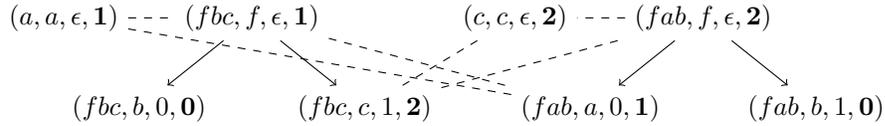
\begin{center}
\begin{tikzpicture}
\node (a) at (1,0) {$(a,a,\epsilon,\mathbf{1})$};
\node (fbc) at (3.5,0) {$(fbc,f,\epsilon,\mathbf{1})$};
\node (fbc0) at (2,-1.2) {$(fbc,b,0,\mathbf{0})$};
\node (fbc1) at (5,-1.2) {$(fbc,c,1,\mathbf{2})$};
\node (c) at (7,0) {$(c,c,\epsilon,\mathbf{2})$};
\node (fab) at (9.5,0) {$(fab,f,\epsilon,\mathbf{2})$};
\node (fab0) at (8,-1.2) {$(fab,a,0,\mathbf{1})$};
\node (fab1) at (11,-1.2) {$(fab,b,1,\mathbf{0})$};

\draw[->] (fbc) to (fbc0);
\draw[->] (fbc) to (fbc1);
\draw[->] (fab) to (fab0);
\draw[->] (fab) to (fab1);

\draw[-,dashed] (a) to (fbc) to (fab0) to (a);
\draw[-,dashed] (c) to (fbc1) to (fab) to (c);
\end{tikzpicture}
\end{center}
The bit in the label representing the type has been bolded for disambiguation purposes. One easily sees that the congruence classes of both $a$ and $c$ are infinite.
For $a$, observe that $a\sim_{\Gamma}fbc\sim_{\Gamma}fbfab\sim_{\Gamma}fbffbcb\sim_{\Gamma}fbffbfabb\sim_{\Gamma}\cdots$.
\end{example}

The reader may already guess that checking for infinite congruence classes can thus be reduced to searching for certain cycles in $R_{\Gamma}$.
We formalize this notion below:

\begin{definition}
Let $\widehat{R_{\Gamma}}$ be the quotient (directed) graph with nodes $[v]$ representing the collection of nodes $w$ in $R_{\Gamma}$ connected to $v$ by an undirected edge.
For two nodes $[v],[w]$ of $\widehat{R_{\Gamma}}$, there is a directed edge with source $[v]$ and target $[w]$ if there is some $v'\in[v]$ and $w'\in[w]$ so that there is a directed edge with source $v'$ and target $w'$ in $R_{\Gamma}$.

We say that a type $i$ is {\bf cyclic} if there exists a directed cycle in $\widehat{R_{\Gamma}}$ that is reachable from a node with type $i$ by a directed path. 
Note that self-loops are included as cycles.
\end{definition}

\begin{theorem}\label{thm:size}
There is a polynomial time algorithm which on input $\Gamma$, a finite set of ground term equations, and $t$ a ground term, outputs the size of the $\sim_{\Gamma}$-congruence class of $t$ (possibly $\infty$).
\end{theorem}

\begin{proof}
Since $\widehat{R_{\Gamma}}$ can evidently be produced in polynomial time in $R_{\Gamma}$, which is itself polynomial in $|\Gamma|$, checking if a type is cyclic (using depth-first search) is also polynomial in $|\Gamma|$.
Then, given a ground term $t$, checking if it contains a subterm of cyclic type is polynomial in $|\Gamma|,|t|$.
If $t$ contains a subterm of cyclic type, then let the witness of this property be the directed walk of nodes $[v_{0}],[v_{1}],\dots,[v_{m}],[v_{m+1}],\dots,[v_{n}]$, where $[v_{n}]=[v_{m}]$.
Recall that each node in $R_{\Gamma}$ represents some ground term.
Then, following the definition of $\widehat{R_{\Gamma}}$, each node $[v]$ represents some collection of terms which are $\F_{\Gamma}$-equivalent.
Similarly, a directed edge with source $[v]$ and target $[w]$ indicates that some term represented by $[w]$ is a strict subterm of a term represented by $[v]$.
In this way, one can easily generate infinitely many distinct ground terms which are each $\F_{\Gamma}$-equivalent to $t$.
This check evidently runs in polynomial time in $|\Gamma|,|t|$ using Kozen's algorithm to detect types.

\smallskip

If on the other hand, no subterm of $t$ is of cyclic type, then we may compute the size of its $\sim_{\Gamma}$-congruence class.
\begin{description}
\item[Computing $\#{[v]}$] Let $v_{0},\dots,v_{n}$ be the nodes of $R_{\Gamma}$ contained in $[v]$.
Compute $r(v_{i})$
for each $i$ and assume w.l.o.g.~that they are distinct.
For each $i$, let $w_{i_{0}},\dots,w_{i_{l_{i}}}$ be the nodes that are typed in $r(v_{i})$.
Note that these are the `closest' typed nodes to $v_{i}$ in $R_{\Gamma}$; in the directed path from $v_{i}$ to $w_{i_{j}}$, no nodes are typed except for $v_{i}$ and $w_{i_{j}}$.
Then taking the empty product to be $1$, define $\#[v]=\sum_{i\leq n}\prod_{j\leq l_{i}}\#[w_{i_{j}}]$.
\end{description}
To see why the algorithm terminates, the key observation is that only non-cyclic types are reachable from non-cyclic types in $\widehat{R_{\Gamma}}$.
Suppose w.l.o.g.~that $[v]$ is of type $1$.
Then, the collection of typed nodes that are reachable from one of $v_{i}\in[v]$ is necessarily not of type $1$, otherwise, $1$ must be a cyclic type.
Repeating the argument, we see that the collection of nodes which are reachable from the typed nodes $[w]$ which are reachable from $[v]$ must also have (non-cyclic) types distinct from $[w]$ and $[v]$.
Then, the maximum depth of the recursion is exactly the total number of types.
Additionally, the breadth of the recursion tree is also bounded by $|\Gamma|$ (the sum of $|p|+|q|$ for each $p=q\in\Gamma$) as each distinct $r(v_{i})$ has size bounded by a distinct term occurring in $\Gamma$.
Thus, the algorithm above is polynomial time computable in $|\Gamma|$.

Now we show that the algorithm is correct.
We will proceed by induction on the depth of the recursion, $d$, required to compute $\#[v]$ for a typed node $[v]$.
For the base case $d=0$, we have that no typed node is reachable from $[v]$ in $\widehat{R_{\Gamma}}$.
Then we evidently have that for each node $v_{i}\in[v]$ of $R_{\Gamma}$, (the term represented by) all strict subterms of $v_{i}$ are not $\F_{\Gamma}$-equivalent to any term mentioned in $\Gamma$.
Thus, the $\sim_{\Gamma}$-congruence class of $v$ is exactly $[v]$.

For typed nodes $[v]$ for which the computation $\#[v]$ requires a recursion depth $d>0$, let $v_{0},v_{1},\dots,v_{n}\in[v]$ be all the typed terms equivalent to $v$, such that $r(v_{i})$ and $r(v_{j})$ are distinct for $i\neq j$.
Now consider a leaf labeled with a type contained in $r(v_{i})$.
Note that for such a type, the depth of recursion required to compute $\#[w]$ for a node $w$ of such a type is strictly less than the depth required to compute $\#[v]$.
Then, by induction, the total number of terms $\F_{\Gamma}$-equivalent to $w$ is given by $\#[w]$.

Now let $t$ be a term that is $\F_{\Gamma}$-equivalent to $v$.
Observe that $r(t)$ demarcates all maximal strict subterms of $t$ that can be replaced to obtain other terms in the $\sim_{\Gamma}$-congruence class of $t$.
Then, for $t$ to be $\F_{\Gamma}$-equivalent to $v$, we necessarily have that $r(t)=r(v_{i})$ for some $i$.
Additionally, since these are distinct, for each $t$ $\F_{\Gamma}$-equivalent to $v$, there is exactly one $i$ so that $r(t)=r(v_{i})$.
Thus, counting the total number of terms $\F_{\Gamma}$-equivalent to $v$ reduces to counting the total number of terms $t$ such that $r(t)=r(v_{i})$ for each $i$.
This is then evidently equal the product of $\#[w]$ over the leafs of $r(v_{i})$ labeled with a type.
For those $v_{i}$ without such leaves, observe that if $r(t)=r(v_{i})$, then $t=v_{i}$, i.e., there is exactly one term $t$ for which $r(t)=r(v_{i})$.

\smallskip

To complete the proof of the converse, let $t$ a ground term be given such that none of its subterms are of cyclic type.
If $t$ has a type, then we are done by applying the algorithm above.
Otherwise, one easily sees that the size of the $\sim_{\Gamma}$-congruence class is given by the sum of $\#[v]$ over the leaves of $r(t)$ labeled with a type for nodes $v$ of the given type.
\end{proof}

\subsection{The intrinsic infinity problem}

Since the next two problems concern global properties of the algebra, it is reasonable to fix some signature $\Sigma$.
From Section \ref{sec:size}, we immediately obtain:

\begin{corollary}
Each $\sim_{\Gamma}$-congruence class is infinite if and only if each constant is of cyclic type.
Thus, for a fixed $\Sigma$, the intrinsic infinity problem is polynomial time decidable in $\Gamma$.
\end{corollary}

\begin{proof}
If all constants are of cyclic type, then every $\sim_{\Gamma}$-congruence class must be infinite.
Conversely, if there is some constant not of cyclic type, then the $\sim_{\Gamma}$-congruence class of this constant is finite.
\end{proof}

\subsection{The finiteness problem}
To facilitate the discussion, we define $ST(\Gamma)$ to be the collection of all subterms $u$ of $p$ where $p=q\in\Gamma$ or $q=p\in\Gamma$ for some $q$.
In order for $\F_{\Gamma}$ to be finite, the intuition is that all terms must eventually `collapse' to terms of smaller depth.
Thus, an algorithm to check if $\F_{\Gamma}$ is finite should search for terms which never `collapse'.
For instance, if there is some term $t$ that is not $\F_{\Gamma}$-equivalent to any term in $ST(\Gamma)$, then we claim that $ftt,\,ffttftt,\,\dots$ are pairwise non-$\F_{\Gamma}$-equivalent (assuming $f$ is binary).
The question then becomes: How many terms should we check?
To bound this search, we work up to $\F_{\Gamma}$-equivalence utilising Kozen's algorithm.
More formally, we prove:

\begin{theorem}\label{thm:finite}
$\F_\Gamma$ is finite if and only if $ST(\Gamma)$ satisfies the following:
(1) All constants from $\Sigma$ are $\F_{\Gamma}$-equivalent to some term in $ST(\Gamma)$; and (2) for every function symbol $f\in\Sigma$, and every $t_1,\dots,t_{\arity(f)}\in ST(\Gamma)$, there exists $u\in ST(\Gamma)$ such that $ft_{1}\dots t_{\arity(f)}\sim_{\Gamma}u$.
That is, $ST(\Gamma)$ is an algebra with signature $\Sigma$.
\end{theorem} 

\begin{proof}
To show the forward implication, suppose that $ST(\Gamma)$ fails to satisfy (1) or (2).
Let $s$ be a term witnessing this failure, that is, $s$ is a term that is not $\F_{\Gamma}$-equivalent to any $u\in ST(\Gamma)$.
Consider the sequence of terms $\{s_n\}_{n\in\N}$ defined as follows: $s_0=s$, and $s_{n+1}=fs_n\dots s_n$.
If $\F_\Gamma$ is finite, then the sequence $s_0,s_1,s_2,\dots$ must be eventually periodic; there exists $m\in\N$ and $p>0$ such that for all $n\geq m$, $\F_{\Gamma}\vDash s_{n}=s_{n+p}$.
However, observe that for any term $t\neq u$ (as terms), if $\F_{\Gamma}\vDash t=u$, then there must be subterms of $t$ and $u$ which are contained in $ST(\Gamma)$.
(Any non-trivial proof of $\F_{\Gamma}$-equivalence must utilise equations from $\Gamma$.)
Since we evidently have that $s_{n}\neq s_{n+p}$ as terms and they respectively have heights $n+\Delta(s)$ and $n+p+\Delta(s)$, we obtain that some subterm of $s_{n+p}$ of height at least $p+\Delta(s)$ is contained in $ST(\Gamma)$, and thus, $s\in ST(\Gamma)$, a contradiction. 

\smallskip

For the converse, we show that if $ST(\Gamma)$ satisfies (1), (2) as in the statement of the theorem, then $ST(\Gamma)\cong\F_{\Gamma}$.
Consider the map $\varphi:\F_{\Gamma}\to ST(\Gamma)$ defined as follows:
For each constant symbol $c\in\Sigma$ define $\varphi(c)$ to be $u\in ST(\Gamma)$ so that $\F_{\Gamma}\vDash c=u$.
Such a $u$ must exist by (1).
Then we extend this map in the obvious way, letting $\varphi(ft_{0}\dots t_{\arity(f)-1})$ be the $u\in ST(\Gamma)$ that is $\F_{\Gamma}$-equivalent to $f\varphi(t_{0})\dots\varphi(t_{\arity(f)-1})$ which exists by (2).
It is easy to see that $\varphi$ is well-defined (up to $\F_{\Gamma}$-equivalence) and an isomorphism.
\end{proof}

Since $ST(\Gamma)$ can be produced in $O(|\Gamma|)$ time, then checking that $ST(\Gamma)$ satisfies the conditions takes $O\big(|\Gamma|^{k+1}\big)$ iterations of Kozen's algorithm, where $k$ is the maximum arity of function symbols in $\Sigma$.

\begin{corollary}
For a fixed $\Sigma$, the finiteness problem is polynomial time decidable.$\hfill\qed$
\end{corollary}

\subsection{The isomorphism problem}\label{sec:iso}
Recall from Theorem \ref{thm:char} that each $\F_{\Gamma}$ is a free algebra over some finite partial algebra $\B$.
Thus, the natural approach to checking if $\F_{\Gamma_{1}}\cong\F_{\Gamma_{2}}$ would be to attempt to extract these finite partial algebras $\B_{1},\B_{2}$ and checking if they are isomorphic.
It is easy to that if $\B_{1}\cong\B_{2}$, then $\F(\B_{1})\cong\F(\B_{2})$.
While the converse does not hold in general, we show:
\begin{theorem}\label{thm:iso}
Given $\Sigma,\Gamma_{1},\Gamma_{2}$, there exists finite partial algebras $\B_{1},\B_{2}$ so that $\F_{\Gamma_{1}}\cong\F_{\Gamma_{2}}$ if and only if $\B_{1}\cong\B_{2}$.
\end{theorem}

\begin{proof}
We first formally define the partial algebras $\B_{1},\B_{2}$.
For a fixed signature $\Sigma$, let $C$ be the set of all constants in $\Sigma$ and $B=C\cup ST(\Gamma_1\cup\Gamma_2)$.
We define the partial algebra induced by $\Gamma_{1}$ as
$\B_1=(B_{1};\{f\}_{f \in \Sigma})$ as follows: 
\begin{enumerate}
    \item $B_{1}=\{[t]_{\Gamma_1}\mid t\in B\}$,
    where $[t]_{\Gamma_1}$ denotes the $\sim_{\Gamma_1}$-congruence class of the term $t$.
    \item For each function symbol $f\in\Sigma$, the partial operation $f^{\B_{1}}: B_{1}^{\arity(f)}\rightharpoonup B_{1}$ is defined as $$f([t_1]_{\Gamma_1},\dots,[t_{\arity(f)}]_{\Gamma_1})^{\B_{1}}
    :=\begin{cases}
    [ft_{1}\dots t_{\arity(f)}]_{\Gamma_1}&\text{if }[ft_{1}\dots t_{\arity(f)}]_{\Gamma_{1}}\in B_{1}\\
    \text{undefined}& \text{otherwise.}
    \end{cases}$$
\end{enumerate}
We define the partial algebra $\B_{2}=(B_{2};\{f\}_{f\in\Sigma})$ induced by $\Gamma_{2}$ mutatis mutandis.

If $\B_{1}\cong\B_{2}$, then we evidently have that $\F(\B_{1})\cong\F(\B_{2})$.
We now show that $\F_{\Gamma_{i}}\cong\F(\B_{i})$ for $i=1,2$.

\begin{claim}
$\F_{\Gamma_{i}}$ is the initial object in the category $\mathscr{C}$ of all algebras with signature $\Sigma$ and containing $\B_{i}$ as a substructure, with arrows given by surjective homomorphisms.
That is, $\F_{\Gamma_{i}}\cong\F(\B_{i})$.
\end{claim}

\begin{proof}[Proof of claim]
By definition of $\B_{i}$, we have that $\B_{i}$ is a substructure of $\F_{\Gamma_{i}}$ and thus $\F_{\Gamma_{i}}\in\mathscr{C}$.
Observe that for any $\A\in\mathscr{C}$, $\A\vDash\Gamma_{i}$, and so by the universal property of $\F_{\Gamma_{i}}$, there exists a unique surjective homomorphism from $\F_{\Gamma_{i}}$ to $\A$, and so $\F_{\Gamma_{i}}$ is initial in $\mathscr{C}$.
Since both $\F_{\Gamma_{i}}$ and $\F(\B_{i})$ are initial in $\mathscr{C}$, they must be isomorphic.
\end{proof}
Thus, if $\B_{1}\cong\B_{2}$, then $\F_{\Gamma_{1}}\cong\F(\B_{1})\cong\F(\B_{2})\cong\F_{\Gamma_{2}}$.

\smallskip

For the converse, assume that $\varphi:\F_{\Gamma_1}\to\F_{\Gamma_2}$ is an isomorphism.
Then for every constant symbol $c\in\Sigma$, $\varphi([c]_{\Gamma_1})=[c]_{\Gamma_2}$.
By induction (see below), one can prove that $\varphi([t]_{\Gamma_1})=[t]_{\Gamma_2}$ holds for each $t\in B$.
Thus the restriction
$\varphi\restriction\B_1$ is an isomorphism between $\B_{1},\B_{2}$.
\end{proof}

\begin{lemma}\label{lem:iso}
$\B_{1}\cong\B_{2}$ if and only if for each pair $u,t\in C\cup ST(\Gamma_{1}\cup\Gamma_{2})$, $u\sim_{\Gamma_1}t$ if and only if $u\sim_{\Gamma_2}t$. 
\end{lemma}

\begin{proof}
For the converse, if for each pair $u,t\in C\cup ST(\Gamma_{1}\cup\Gamma_{2})$, $u\sim_{\Gamma_1}t$ if and only if $u\sim_{\Gamma_2}t$, one can verify that $\varphi:[s]_{\Gamma_1}\in\B_{1}\mapsto[s]_{\Gamma_2}\in\B_{2}$ is a well-defined isomorphism.

Now suppose that $\varphi:\B_{1}\to\B_{2}$ is an isomorphism.
We instead show that for any $[u]_{\Gamma_{1}}\in\B_{1}$, $\varphi([u]_{\Gamma_{1}})=[u]_{\Gamma_{2}}$.
Since $\varphi$ is an isomorphism, we must have that for any constant $c_{i}$, $\varphi([c_{i}]_{\Gamma_{1}})=[c_{i}]_{\Gamma_{2}}$.
Now let $u=fu_{1}\dots u_{\arity(f)}$ be given such that $[u]_{\Gamma_{1}}\in\B_{1}$.
Then
\begin{align*}
\varphi([u]_{\Gamma_{1}})&=\varphi([fu_{1}\dots u_{\arity(f)}]_{\Gamma_{1}})\\
&=\varphi\big(f([u_{1}]_{\Gamma_{1}},\dots,[u_{\arity(f)}]_{\Gamma_{1}})^{\B_{1}}\big)&&(\text{by definition of }\B_{1})\\
&=f(\varphi([u_{1}]_{\Gamma_{1}}),\dots,\varphi([u_{\arity(f)}]_{\Gamma_{1}}))^{\B_{2}}&&(\varphi\text{ is an isomorphism})\\
&=f([u_{1}]_{\Gamma_{2}},\dots,[u_{\arity(f)}]_{\Gamma_{2}})^{\B_{2}}&&(\text{by induction})\\
&=[fu_{1}\dots u_{\arity(f)}]_{\Gamma_{2}} &&(\text{by definition of }\B_{2}).
\end{align*}
Thus, we have that for any $[u]_{\Gamma_{1}}\in\B_{1}$, $\varphi([u_{i}]_{\Gamma_{1}})=[u_{i}]_{\Gamma_{2}}$.
By applying the assumption that $\varphi$ is bijective, we then obtain that $u_{i}\sim_{\Gamma_{1}}u_{j}$ if and only if $u_{i}\sim_{\Gamma_{2}}u_{j}$ as desired.
\end{proof}

The sizes of $\B_{1}$ and $\B_{2}$ are each $O(|\Sigma|+|\Gamma_1|+|\Gamma_2|)$.
By Lemma \ref{lem:iso}, and Kozen's algorithm, checking if $\B_{1}\cong\B_{2}$ is polynomial time in $|\Gamma_{1}|,|\Gamma_{2}|,|u|,|t|$.
Thus we obtain:
\begin{corollary}
The isomorphism problem is polynomial time decidable in $|\Sigma|,|\Gamma_{1}|,|\Gamma_{2}|$.$\hfill\qed$
\end{corollary}

\section{Quantifier elimination for almost free algebras}\label{sec:qe}

\begin{definition}\label{def:is}
For a function symbol $f$, define the predicate $\is{f}$ as follows.
If $a\in\B$, then define $\is{f}(a)=\bot$.
If $a\notin\B$, then define $\is{f}(a)=\top$ if $a=ft_{1}\dots t_{\arity(f)}$ for some $\F(\B)$-terms $t_{1},\dots,t_{\arity(f)}$, and define $\is{f}(a)=\bot$ otherwise.
\end{definition}

Let $\F(\B)^{*}$ be the expansion of $\F(\B)$ with the {\bf tester predicates} $\is{f}$ for each function symbol $f$ in the language of $\B$.

\begin{definition}\label{def:standardformula}
A formula is {\bf special} if it is of the following form:
$$
\exists y_{0}\dots y_{m}\left(\bigwedge_{i}x_{\alpha_{i}}=t_{\alpha_{i}}\wedge\bigwedge_{j}x_{\beta_{j}}\neq t_{\beta_{j}}\wedge\bigwedge_{k}y_{\gamma_{k}}\neq t_{\gamma_{k}}\wedge\bigwedge_{l}\is{f_{\delta_{l}}}(y_{\varepsilon_{l}})\wedge\bigwedge_{r}\neg\is{f_{\xi_{r}}}(y_{\chi_{r}})\right),
$$
where each $x_{\alpha_{i}}$ occurs exactly once in the formula, and the $t$'s are $\F(\B)$-terms.
Additionally, in each equation $z=t$ or non-equation $z\neq t$, where $z$ is a variable, $z$ does not occur in $t$.

We say that a formula is {\bf standard} if it is constructed from special formulas, quantifier-free $\F(\B)^{*}$-formulas, and the logical connectives $\vee,\wedge$.
\end{definition}

The bulk of the proof will be to show that all $\F(\B)^{*}$-formulas can be rewritten into an equivalent standard formula.
This will be split into two main steps: (1) Lemma \ref{lem:eformula} will show that all existential formulas are standard; (2) and Lemma \ref{lem:negstandard} shows that standard formulas are closed under negation.
With these two lemmas, given any $\F(\B)^{*}$-formula of the form $\forall y\varphi$ where $\varphi$ is standard (i.e, may not be quantifier-free), we obtain that $\neg(\exists y\neg\varphi)$ is again standard by applying Lemma \ref{lem:negstandard}, \ref{lem:eformula}, and \ref{lem:negstandard} in order.

\begin{lemma}\label{lem:eformula}
Let $\varphi$ be a formula constructed from quantifier-free $\F(\B)^{*}$-formulas, the logical connectives $\vee,\wedge$, and the existential quantifier.
Then $\varphi$ is $\F(\B)^{*}$-equivalent to a standard formula.
\end{lemma}

\begin{proof}
First, we extract all (existential) quantifiers from $\varphi$ and write the resulting formula in disjunctive normal form.
Then, by distributing the existential quantifiers over the disjunctions, it suffices to consider a formula $\exists y_{0}\dots y_{m}\psi$, where $\psi$ is quantifier-free and a conjunction of equations and non-equations of $\F(\B)$, and predicates $\is{f}(z)$ and $\neg\is{f}(z)$ where $z$ is a variable.

Now, rewrite $\exists y_{0}y_{1}\dots y_{m}\psi$ as the equivalent $\psi_{0}\wedge\exists y_{0}\dots y_{m}\psi_{1}$ where $\psi_{0}$ does not mention any of the bound variables and each conjunct in $\psi_{1}$ has at least one occurrence of one of the bound variables.
Observe that $\psi_{0}$ is already a standard formula; any quantifier-free $\F(\B)^{*}$-formula is standard.
Now we will rewrite $\psi_{1}$ as a disjunction of special formulas.

\paragraph*{Step one: processing the equations}
We first ensure that every equation occurring in $\psi_{1}$ is of the form $x=t$ where $x$ is a variable distinct from the bound variables $y_{0},y_{1},\dots,y_{m}$.
Additionally, each equation $x=t$ is also such that $x$ does not occur in $t$.
For each equation occurring in $\psi_{1}$, we do the following:
\begin{enumerate}[label=(1\roman*)]
\item If the equation is $u=t$ with $\Delta(u)=\Delta(t)=0$, then at least one of $u$ or $t$ should be a bound variable.\label{step:1i}
Suppose w.l.o.g.~that $u$ is the bound variable $y_{0}$.
If $t$ is also the bound variable $y_{0}$, then we can omit the equation $u=t$ from $\psi_{1}$.
Otherwise, we may rewrite $\exists y_{0}y_{1}\dots y_{m}\psi_{1}$ as $\exists y_{1}\dots y_{m}\psi_{1}[y_{0}/t]$.

\smallskip

\item If the equation is $u=t$ with $\min\{\Delta(u),\Delta(t)\}=0$, then w.l.o.g.~suppose that $\Delta(u)=0$ and $\Delta(t)>0$.
If we further have that $u$ is a constant, then at least one of the bound variables $y_{0},y_{1},\dots,y_{m}$ must occur in $t$, say $y_{0}$.
Now, in order for the equation $u=t$ to hold, we necessarily have that $y_{0}\in\B$.
Then, we may rewrite $\exists y_{0}y_{1}\dots y_{m}\psi_{1}$ as
$$
\bigvee_{b\in B}(\exists y_{1}\dots y_{m}\psi_{1}[y_{0}/b]).
$$
In particular, there are no longer any equations in $\psi_{1}[y_{0}/b]$ in which $y_{0}$ occurs.\label{step:1ii}

Now consider the case where $u$ is not a constant.
Then let $u$ be a variable, say $z$.
If $z$ also occurs in $t$, then it must be that $z$ must take a value from $\B$.
In which case, all subterms of $t$ necessarily must also take values from $\B$.
Then, by letting $z_{0},z_{1},\dots,z_{k}$ be all the variables which occur in $u=t$, we may rewrite the equation as
\begin{equation}\label{eq:one}
\bigvee_{(b_{0},b_{1},\dots,b_{k})\in\B^{k+1}}\left((u=t)[z_{0}/b_{0},\dots,z_{k}/b_{k}]\wedge\bigwedge_{i\leq k}z_{i}=b_{i}\right).
\end{equation}
Each of the equations $(u=t)[z_{0}/b_{0},\dots,z_{k}/b_{k}]$ can be evaluated and thus replaced with either $\bot$ or $\top$, effectively removing either the relevant disjunct or the equation $(u=t)[z_{0}/b_{0},\dots,z_{k}/b_{k}]$ from $\psi_{1}$ respectively.
That is, this rewriting only introduces new equations $u'=t'$ for which $\Delta(u')=\Delta(t')=0$, which can in turn be processed in the same way as in \ref{step:1i}.

Finally, if $z$ does not occur in $t$, then we consider two further sub-cases.
If $z$ is not one of the bound variables, then we can leave $u=t$ unchanged in $\psi_{1}$.
Note that we only wish to remove equations of the form $y_{i}=t$ (recall Definition \ref{def:standardformula}).
Thus, it remains to consider the case when $u$ is a bound variable, say $y_{0}$.
By assumption, we already have that $y_{0}$ does not occur in $t$, and so, $\exists y_{0}y_{1}\dots y_{m}\psi_{1}$ can be rewritten as $\exists y_{1}\dots y_{m}\psi_{1}[y_{0}/t]$.
In this new formula, $y_{0}$ no longer occurs and thus no equations containing $y_{0}$ can occur.

We remark that in this last step, the maximum height of terms contained in $\psi_{1}[y_{0}/t]$ will generally be larger than the maximum height of terms in $\psi_{1}$.
Nonetheless, when such a replacement is made, we also drop one existential quantifier; the height of terms in our formula thus cannot increase arbitrarily many times for this reason.

\smallskip

\item If the equation is $u=t$ with $\min\{\Delta(u),\Delta(t)\}>0$, then let $u=fu_{0}u_{1}\dots u_{\arity(f)-1}$ and $t=gt_{0}t_{1}\dots t_{\arity(g)-1}$.
Unlike the case for term algebras, if $f\neq g$, we cannot directly conclude that it is a contradiction.
Similarly, even if $f=g$, this also does not imply that $u_{i}=t_{i}$ for each $i<\arity(f)$.
For instance, it could be that both $u,t$ evaluate to the same value in $\B$, but in different ways.
Thus, we do the following instead.\label{step:1iii}

Let $z_{0},z_{1},\dots,z_{k}$ be all the variables occurring in $u=t$.
Then we may rewrite $u=t$ simply as (\ref{eq:one}) if $f\neq g$ and
$$
(\ref{eq:one})\vee\bigvee_{j<\arity(f)}\left(u_{j}\notin \B\wedge\bigwedge_{i<\arity(f)}u_{i}=t_{i}\right)
$$
otherwise.
To see why the formula above is equivalent to $u=t$, observe that if at least one of the subterms $u_{j}$ is not contained in $\B$ (expressed as $\bigwedge_{b\in\B}u_{j}\neq b$), then $u$ and thus $t$, must also both be outside of $\B$.
In such a case, for $u=t$ to be satisfied, we obtain that $f=g$, and $u_{i}=t_{i}$ for each $i<\arity(f)$.
On the other hand, if each of $u_{i}$ is contained in $\B$, then all of their subterms must correspondingly also be contained in $\B$.
In particular, we may simply substitute all possible values from $\B$ into the variables $z_{0},z_{1},\dots,z_{k}$ to evaluate the equation $u=t$ as expressed by (\ref{eq:one}).
Furthermore, if $f\neq g$, then $u=t$ can only be satisfied if each of $u_{i}$ is contained in $\B$, and thus, in such a case, $u=t$ is simply equivalent to (\ref{eq:one}).

Just as before, all equations containing no variables can be evaluated, and the relevant disjuncts or equations can be removed from the formula.
Note that all terms in this new formula have height $<\min\{\Delta(u),\Delta(t)\}$.
Thus, by repeating this process, and the steps in \ref{step:1i}, \ref{step:1ii}, we eventually end up with a formula in which all equations are of the form $x=t$ where $x$ is a free variable.
\end{enumerate}

In \ref{step:1i},\ref{step:1ii},\ref{step:1iii} above, we often rewrite $\psi_{1}$ in disjunctive normal form, but each time we do, we distribute the existential quantifiers just as in the beginning of the proof.
Thus, it again suffices to consider only formulas of the form $\exists y_{0}y_{1}\dots y_{m}\psi_{2}$ where $\psi_{2}$ is a conjunction such that every equation in $\psi_{2}$ is of the form $x=t$ for some $x$ distinct from the bound variables and $x$ does not occur in $t$.
Recall from Definition \ref{def:standardformula} that we want each such free variable $x$ to occur exactly once in the formula.
The next step shall ensure this.

\paragraph*{Step two: removing repetition of free variables}
Consider an equation $x=t$ occurring in $\psi_{2}$ and let $\psi_{2}'$ be the formula $\psi_{2}$ without the equation $x=t$; $\psi_{2}\equiv x=t\wedge\psi_{2}'$.
\begin{enumerate}[label=(2\roman*)]
\item If no equation in $\psi_{2}'$ has $x$ as the subject, then we rewrite $\psi_{2}$ as $x=t\wedge(\psi_{2}'[x/t])$.
By our assumption that $x$ does not occur in $t$, we now have that the only occurrence of $x$ in the rewritten $\psi_{2}$ is exactly as the subject of $x=t$.
Additionally, all equations now occurring in the rewritten $\psi_{2}$ still remains of the form $x'=t'$ where $x'$ is free.
However, we note that it may now be that $x'$ does occur in $t'$.\label{step:2i}

Now, if we see such an equation failing to satisfy this property, then we apply \ref{step:1ii} to such an equation, in particular, the rewriting done in (\ref{eq:one}).
This introduces multiple new equations of the form $z_{i}=b_{i}$ where each $z_{i}$ is a variable occurring in $x'=t'$.
Nonetheless, by the actions performed above, all such $z_{i}$ cannot be $x$.
That is, once we further rewrite $\psi_{2}$ by applying (\ref{eq:one}) to the relevant equations, $x=t$ remains the only occurrence of $x$ within the rewritten $\psi_{2}$, and it now satisfies the properties assumed at the beginning of this step.

\smallskip

\item Now suppose that $\psi_{2}'$ has at least one other equation with $x$ as the subject, say $x=u$.
Then we perform the same actions as in \ref{step:2i}, rewriting $\psi_{2}$ as $x=t\wedge(\psi_{2}'[x/t])$.
As before, the equation $x=t$ is now the only occurrence of $x$ in this rewritten $\psi_{2}$.
For the equations $x'=t'$ where $x'$ now occurs in $t'$ due to the rewriting, we perform the same actions as in the previous step.\label{step:2ii}

The main difference from \ref{step:2i} is that we may now have equations of the form $u'=t'$ where $\min\{\Delta(u'),\Delta(t')\}>0$.
For such equations, we perform the actions as described in \ref{step:1iii}.
Such a rewriting introduces multiple new equations, but we note again that these new equations only mention variables occurring in $u'=t'$.
That is, no new equations mentioning the free variable $x$ is introduced.
\end{enumerate}
By repeating the steps above for each free variable $x$ occurring in $\psi_{2}$, we eventually obtain that in each disjunct, if $x=t$ occurs, where $x$ is free, then this equation is the only occurrence of $x$ in the disjunct as desired.

\paragraph*{Step three: processing the non-equations}
We now process the non-equations.
Recall from Definition \ref{def:standardformula} that we want each non-equation to be of the form $z\neq t$ where $z$ is a variable and such that $z$ does not occur in $t$.
Now consider each non-equation $u\neq t$ occurring in the rewritten formula $\psi_{2}$.
Recall also that each such non-equation should contain some bound variable, otherwise we can add it to $\psi_{0}$.
\begin{enumerate}[label=(3\roman*)]
\item If $\Delta(u)=\Delta(t)=0$, then either the truth value of $u\neq t$ can be evaluated immediately, or it must already be of the desired form.\label{step:3i}

\smallskip

\item If $\min\{\Delta(u),\Delta(t)\}=0$, then assume w.l.o.g.~that $\Delta(u)=0<\Delta(t)$.
If $u$ is a constant, then $t$ must contain some bound variable, say $y_{0}$.
Now, in order for $u\neq t$ to hold, we either have $y_{0}\notin\B$, which implies that any evaluation of $t$ also cannot be contained in $\B$, or $y_{0}$ is some value in $\B$ but yet $t$ never evaluates to $u$ under all possible assignments.
In other words, $u\neq t$ can be rewritten as\label{step:3ii}
\begin{equation}\label{eq:two}
y_{0}\notin\B\vee\bigvee_{b\in\B}\big(u\neq t[y_{0}/b]\wedge y_{0}=b\big).
\end{equation}

After expanding and rewriting our formula in disjunctive normal form, note that we have now introduced new equations $y_{0}=b$ for various $b\in\B$ which mention the bound variable $y_{0}$.
Nonetheless, by applying \ref{step:1i} wherever necessary, all such equations are removed from our formula.
Furthermore, in doing so, the only new non-equations introduced into the formula are $y_{0}\neq b$ for various $b\in\B$ (recall $y_{0}\notin\B$ is represented as $\bigwedge_{b\in\B}y_{0}\neq b$) and $u\neq t[y_{0}/b]$.
Observe that $y_{0}\neq b$ is already in the desired form and thus no further processing is required.
For the non-equations $u\neq t[y_{0}/b]$, these have one less bound variable occurring in them than $u\neq t$.
Thus, by repeating this process, we eventually obtain a formula with no variables, whose truth value can be evaluated, or a formula with no occurrence of the bound variables, in which case it can be moved to $\psi_{0}$.

Now suppose that $u$ is a variable, say $z$.
If $z$ does not occur in $t$, then we are done as $u\neq t$ is already of the desired form.
If $z$ occurs in $t$, then for $z\neq t$ to be satisfied, we again have that either $z\notin\B$, in which case it cannot be a strict subterm of itself, or that $z\in\B$ but all assignments of values from $\B$ to $z$ results in $z\neq t$.
Then we may rewrite $z\neq t$ just as in (\ref{eq:two}), but replacing $y_{0}$ with $z$.
Once again, this introduces new equations $z=b$ and non-equations $z\neq b$ and $(z\neq t)[z/b]$.
These new equations can then be processed just as in \ref{step:1i} if $z$ is a bound variable, or \ref{step:2i},\ref{step:2ii} if $z$ is free.
Just as in the previous step, the non-equations $(z\neq t)[z/b]$ introduced in this rewriting are either already of the desired form or have one less variable in them than $u\neq t$. 
Thus, the process must eventually terminate.

\smallskip

\item Finally, if $\min\{\Delta(u),\Delta(t)\}>0$, then let $u=fu_{0}\dots u_{\arity(f)-1}$ and $t=gt_{0}\dots t_{\arity(g)-1}$.
Also let $z_{0},z_{1},\dots,z_{k}$ be all the variables occurring in $u\neq t$.
Now, we separately consider the cases when $f\neq g$ and when $f=g$.\label{step:3iii}

If $f\neq g$, then $u\neq t$ is satisfied when one of the following holds: $u_{i}\notin\B$ for some $i$; $t_{j}\notin\B$ for some $j$; all of $u_{i}$ and $t_{j}$ are contained in $\B$ but the terms never evaluate to the same value.
This can then be written syntactically as:
\begin{equation*}
\bigvee_{i<\arity(f)}u_{i}\notin\B\vee\bigvee_{j<\arity(g)}t_{j}\notin\B\vee\bigvee_{(b_{0},\dots,b_{k})\in\B^{k+1}}\left((u\neq t)[z_{0}/b_{0},\dots,z_{k}/b_{k}]\wedge\bigwedge_{l\leq k}z_{l}=b_{l}\right).
\end{equation*}
To see why it is equivalent to $u\neq t$, observe that if there is some $i$ such that $u_{i}\notin\B$, then $u\notin\B$.
That is, we cannot have $u=t$ since they begin with different function symbols.
The same argument holds if there is some $j$ such that $t_{j}\notin\B$.
On the other hand, if each of $u_{i},t_{j}$ are all contained in $\B$, then the truth value of $u\neq t$ boils down to evaluating $u\neq t$ under each possible assignment of values to the variables occurring in $u\neq t$.
Thus, the formula above is equivalent to $u\neq t$ if $f\neq g$.

Observe that all non-equations occurring in the above formula either contains no variables, or can be further be processed by \ref{step:3i} or \ref{step:3ii}.
The equations $z_{l}=b_{l}$ can again be processed by \ref{step:1i}, \ref{step:2i}, or \ref{step:2ii}.
As mentioned above, these further actions possibly causes yet other actions to happen but must eventually terminate.

If $f=g$, then we have a slight complication in that even if $u_{i}\notin\B$ or $t_{j}\notin\B$ holds, we cannot immediately conclude that $u\neq t$.
To obtain that $u\neq t$, we must further add in a clause to assert that at least one of $u_{l}\neq t_{l}$; since $u\notin\B$, then the (non-)equality of $u,t$ boils down to the (non-)equality of their subterms.
Thus, if $f=g$, $u\neq t$ is equivalent to the following:
\begin{equation*}
\begin{split}
&\bigvee_{i<\arity(f)}\left(\big(u_{i}\notin\B\vee t_{i}\notin\B\big)\wedge\bigvee_{j<\arity(f)}u_{j}\neq t_{j}\right)\\
&\quad\vee\bigvee_{(b_{0},\dots,b_{k})\in\B^{k+1}}\left((u\neq t)[z_{0}/b_{0},\dots,z_{k}/b_{k}]\wedge\bigwedge_{l\leq k}z_{l}=b_{l}\right).
\end{split}
\end{equation*}
Similar observations and arguments as before allows one to conclude that the actions required to process these newly introduced equations and non-equations do not cause any non-terminating behaviour.
\end{enumerate}

It follows from the steps taken that once the process terminates, we obtain a formula which is a disjunction of conjunctions of formulas in which every equation and non-equation is of the following forms: \ref{step:1i}, \ref{step:1ii}, \ref{step:1iii} ensures all equations are of the form $x=t$ where $x$ is free and $x$ does not occur in $t$; \ref{step:2i}, \ref{step:2ii} ensures that for each equation $x=t$ occurring in a disjunct, $x$ occurs nowhere else in the same disjunct; and finally, \ref{step:3i}, \ref{step:3ii}, \ref{step:3iii} ensures that all non-equations are of the form $z\neq t$ where $z$ is a variable which does not occur in $t$.

\paragraph*{Step four: processing the tester predicates}
Finally, in the remaining step, we will process all the predicates $\neg\is{f}(t)$ and $\is{f}(t)$ occurring in our formula.
To ensure that our rewritten formula is a standard formula, we must remove all occurrences of $\neg\is{f}(t)$ and $\is{f}(t)$ from our formula where $t$ is not a variable.
Now, let $\psi_{3}$ be a disjunct in the rewritten $\psi_{2}$ satisfying the above.
We will only describe the process for the predicates $\neg\is{f}(t)$ as one can easily extract a process for $\is{f}(t)$ from it.
Fix a formula $\neg\is{f}(t)$ occurring in $\psi_{3}$ and do the following.
\begin{enumerate}[label=(4\roman*)]
\item If $\Delta(t)=0$, then $t$ is either a variable or a constant.
If $t$ is a variable, then we are done.
On the other hand, if $t$ is a constant, then $\neg\is{f}(t)$ evaluates to $\top$ and can thus be omitted from $\psi_{3}$.

\smallskip

\item If $t$ contains variables and $\Delta(t)>0$, then let these variables be $z_{0},z_{1},\dots,z_{k}$.
Evidently, we can no longer directly evaluate the truth value of $\neg\is{f}(t)$.
Instead, we turn to the non-equations occurring in $\psi_{3}$.
If we have the formula $z_{i}\notin\B$ for some $i$ occurring in $\psi_{3}$ (recall that this is the conjunction $\bigwedge_{b\in\B}z_{i}\neq b$), then we have that the term $t$ cannot evaluate to a value in $\B$.
Then the truth value of $\neg\is{f}(t)$ can be evaluated based on the starting symbol of $t$ (see Definition \ref{def:is}).

On the other hand, if for all $i$, $z_{i}\notin\B$ does not occur in $\psi_{3}$, then we need to rewrite the formula $\psi_{3}$ but accounting for all possible values of $t$ as follows.
$$
\bigvee_{(b_{0},\dots,b_{k})\in\B^{k+1}}\left(\neg\is{f}(t[z_{0}/b_{0},\dots,z_{k}/b_{k}])\wedge\bigwedge_{i\leq k}z_{i}=b_{i}\right)\vee\bigvee_{i<k}\big(z_{i}\notin\B\wedge\neg\is{f}(t)\big).
$$
It is easy to see that the above formula is equivalent to $\neg\is{f}(t)$.
To see why it no longer contains $\neg\is{f}(t)$, note that once $z_{i}\notin\B$ has been specified by the formula, the truth value of $\neg\is{f}(t)$ can be evaluated based on the starting symbol of $t$ just as before.
Similarly, since $t[z_{0}/b_{0},\dots,z_{k}/b_{k}]$ contains no variables, we may also evaluate the truth value of $\neg\is{f}(t[z_{0}/b_{0},\dots,z_{k}/b_{k}])$, effectively removing all occurrences of $\neg\is{f}(t)$ from the formula above.

Evidently, this rewriting introduces new equations and non-equations of the forms $z_{i}=b_{i}$ and $z_{i}\neq b$ respectively.
We may then have to further process the equations using \ref{step:1i}, \ref{step:2i}, or \ref{step:2ii}, but those actions never introduce any predicates of the form $\neg\is{f}(u)$.
Moreover, the non-equations introduced in this rewriting of $\psi_{3}$ are all of the desired form and so, we again have that this cannot cause any non-terminating behaviour.
\end{enumerate}

Once all occurrences of $\neg\is{f}(t)$ and $\is{f}(t)$ have been processed, it is easy to see that our formula is now finally a standard formula.
\end{proof}

\begin{remark}\label{rem:elength}
The rewritten formula has length $O\big(|\B|^{|\varphi|}|\varphi|\big)$.
\end{remark}

The next lemma shows that negations of standard formulas are also standard.

\begin{lemma}\label{lem:negstandard}
Let $\varphi$ be a standard formula.
Then $\neg\varphi$ can be written as an $\F(\B)^{*}$-equivalent standard formula.
\end{lemma}

\begin{proof}
Since the negations of quantifier-free $\F(\B)^{*}$-formulas are evidently standard, together with Lemma \ref{lem:eformula}, it suffices to show that negations of special formulas are $\F(\B)^{*}$-equivalent to an existential $\F(\B)^{*}$-formula.

Consider the negation of an arbitrary special formula
\begin{equation}\label{eq:sp}
\forall y_{0}\dots y_{m}\left(\bigvee_{i}x_{\alpha_{i}}\neq t_{\alpha_{i}}\vee\bigvee_{j}x_{\beta_{j}}=t_{\beta_{j}}\vee\bigvee_{k}y_{\gamma_{k}}=t_{\gamma_{k}}\vee\bigvee_{l}\neg\is{f_{\delta_{l}}}(y_{\varepsilon_{l}})\vee\bigvee_{r}\is{f_{\xi_{r}}}(y_{\chi_{r}})\right).
\end{equation}
We shall first ensure that every non-equation $u\neq t$ above is so that $\Delta(u)=\Delta(t)=0$.
Fix a non-equation $x\neq t$ (we suppress the indices) occurring in the formula above.
If $\Delta(t)=0$, then we are done.
Otherwise, let $t=ft_{0}t_{1}\dots t_{\arity(f)-1}$.
Now we split the formula into the disjuncts over whether $\neg\is{f}(x)$ or its negation holds; rewrite (\ref{eq:sp}) as $\big(\is{f}(x)\wedge(\ref{eq:sp})\big)\vee\big(\neg\is{f}(x)\wedge(\ref{eq:sp})\big)$.
Again, we recall that even if $\neg\is{f}(x)$ holds, $x\neq t$ is not a tautology, and similarly, even if $\is{f}(x)$ holds, it does not mean that $x$ must `decompose' into terms $u_{0},u_{1},\dots,u_{\arity(f)-1}$ which are each equal to $t_{i}$.
Let $k=\arity(f)-1$, $\psi$ be the formula (\ref{eq:sp}) but without the quantifiers and the non-equation $x\neq t$, and further rewrite $\big(\is{f}(x)\wedge(\ref{eq:sp})\big)\vee\big(\neg\is{f}(x)\wedge(\ref{eq:sp})\big)$ as
\begin{equation}\label{eq:biggg}
\begin{split}
&\Bigg\{\is{f}(x)\wedge\exists z_{0}\dots z_{k}\Bigg[x=fz_{0}\dots z_{k}\wedge\Bigg(\forall y_{0}\dots y_{m}\Bigg[\psi\vee\bigvee_{i\leq k}\Bigg(\big(z_{i}\notin\B\vee t_{i}\notin\B\big)\wedge\bigvee_{j\leq k}z_{j}\neq t_{j}\Bigg)\\
&\vee\bigvee_{\substack{(b_{0},\dots,b_{k})\in\B^{k+1}\\(c_{0},\dots,c_{k})\in\B^{k+1}}}\Bigg(fb_{0}\dots b_{k}\neq fc_{0}\dots c_{k}\wedge\bigwedge_{i\leq k}\big(z_{i}=b_{i}\wedge t_{i}=c_{i}\big)\Bigg)\Bigg]\Bigg)\Bigg]\Bigg\}\\
&\vee\Bigg\{\neg\is{f}(x)\wedge\Bigg(x\notin\B\vee\forall y_{0}\dots y_{m}\Bigg[\psi\vee\bigvee_{(b,b_{0},\dots,b_{k})\in\B^{k+2}}\Bigg(b\neq fb_{0}\dots b_{k}\wedge\bigwedge_{j\leq k}t_{j}=b_{j}\Bigg)\Bigg]\Bigg)\Bigg\}.
\end{split}
\end{equation}

Now we argue that the above formula is equivalent to $\big(\is{f}(x)\wedge(\ref{eq:sp})\big)\vee\big(\neg\is{f}(x)\wedge(\ref{eq:sp})\big)$.
For the first disjunct (in the curly parentheses), if we have $\is{f}(x)$, then we necessarily have that $x\notin\B$ and may thus be rewritten as $fz_{0}\dots z_{k}$ for some $z_{0},\dots,z_{k}$.
Then, $\bigvee_{i\leq k}\left(\big(z_{i}\notin\B\vee t_{i}\notin\B\big)\wedge\bigvee_{j\leq k}z_{j}\neq t_{j}\right)$ covers the cases where one of $z_{i}$ or $t_{i}$ is assigned some value outside of $\B$, and thus, $x\neq t$ is equivalent to $\bigvee_{j\leq k}z_{j}\neq t_{j}$.
On the other hand, if each of $z_{i}$ and $t_{i}$ are all assigned values from $\B$, then $x\neq t$ is equivalent to $fb_{0}\dots b_{k}\neq fc_{0}\dots c_{k}$ where $b_{i}$ and $c_{i}$ are the assigned values of $z_{i}$ and $t_{i}$ respectively.
Equivalence of the second disjunct (in the curly parentheses) to $\neg\is{f}(x)\wedge (\ref{eq:sp})$ follows by a similar argument and the observation that if $\neg\is{f}(x)$ and $x\notin\B$ holds, then (\ref{eq:sp}) is equivalent to $\bot$.

Observe that each newly introduced non-equation in the formula above either contains no variables, or is such that its height is $<\Delta(t)$.
Thus, by repeating this process, we eventually end up with a formula where all non-equations have height $0$.
In particular, we may assume that we are considering a formula of the form (\ref{eq:sp}) but one where each non-equation $x_{\alpha_{i}}\neq t_{\alpha_{i}}$ is such that $\Delta(t_{\alpha_{i}})=0$.
Furthermore, we may assume that each $t_{\alpha_{i}}$ is one of the bound variables, otherwise the non-equation can be moved out of the scope of the quantifiers $\forall y_{0}\dots y_{m}$.

\begin{remark}\label{rem:expblowup}
Once all non-equations in the formula are of the form $x\neq t$ where $\Delta(t)=0$, observe that for each non-equation $x_{\alpha_{i}}\neq t_{\alpha_{i}}$ originally occurring in the formula, and each strict subterm of $t_{\alpha_{i}}$, we add a new existential quantifier to the formula.
In particular, if $\B$ contains at least one binary function, then the number of existential quantifiers is exponential in the length of the given formula.
\end{remark}

For a bound variable $y_{j}$, pick an arbitrary non-equation $x_{\alpha_{i}}\neq y_{j}$ mentioning it if it exists within our formula.
It is easy to verify that for any formula $\varphi$, $\vDash(\forall y(x\neq y\vee\varphi))\leftrightarrow\varphi[y/x]$.
That is, we may remove the non-equation $x_{\alpha_{i}}\neq y_{j}$ from our formula, and replace all other occurrences of $y_{j}$ with $x_{\alpha_{i}}$.
Then, any other non-equation which previously mentioned $y_{j}$ now no longer mention any bound variable and can thus be moved out of the scope of the quantifiers $\forall y_{0}\dots y_{m}$.
Thus, it remains to consider formulas of the form
$$
\forall y_{0}\dots y_{m}\left(\bigvee_{j}x_{\beta_{j}}=t_{\beta_{j}}\vee\bigvee_{k}y_{\gamma_{k}}=t_{\gamma_{k}}\vee\bigvee_{l}\neg\is{f_{\delta_{l}}}(y_{\varepsilon_{l}})\vee\bigvee_{r}\is{f_{\xi_{r}}}(y_{\chi_{r}})\right).
$$
After a simple rewriting, we obtain
$$
\forall y_{0}\dots y_{m}\left(\left(\bigwedge_{l}\is{f_{\delta_{l}}}(y_{\varepsilon_{l}})\wedge\bigwedge_{r}\neg\is{f_{\xi_{r}}}(y_{\chi_{r}})\right)\rightarrow\left(\bigvee_{j}x_{\beta_{j}}=t_{\beta_{j}}\vee\bigvee_{k}y_{\gamma_{k}}=t_{\gamma_{k}}\right)\right).
$$
By letting $T_{j}$ be the set of elements of $\F(\B)^{*}$ which satisfy the predicates $\is{f_{\delta_{l}}}(y_{\varepsilon_{l}})$ and $\neg\is{f_{\xi_{r}}}(y_{\chi_{r}})$ which mention $y_{j}$, the formula can further be rewritten as
$$
\forall y_{0}\in T_{0}\dots\forall y_{m}\in T_{m}\left(\bigvee_{j}x_{\beta_{j}}=t_{\beta_{j}}\vee\bigvee_{k}y_{\gamma_{k}}=t_{\gamma_{k}}\right).
$$

Now, we claim that if $T_{m}$ is infinite\footnote{We clarify here that we are thinking of elements of $T_{m}$ up to the equivalence induced by $\B$; i.e., we are assuming here that $T_{m}$ contains infinitely many distinct equivalence classes of $\F(\B)$-ground terms.}, then all equations mentioning $y_{m}$ in the formula above can be omitted.
Conversely, if $T_{m}$ is finite, then $T_{m}=\B$, or is an empty set, and so the formula can be rewritten as a finite conjunction (resp.~empty), substituting $y_{m}$ with a value from $\B$ in each conjunct.
(We adopt the convention that an empty conjunction is a tautology and the empty disjunction is a contradiction.)

Suppose w.l.o.g.~that all equations $x_{\beta_{j}}=t_{\beta_{j}}$ and $y_{\gamma_{k}}=t_{\gamma_{k}}$ mention $y_{m}$, otherwise they can be moved beyond the scope of $\forall y_{m}\in T_{m}$.
Evidently, if any of the remaining equations $x_{\beta_{j}}=t_{\beta_{j}}$ or $y_{\gamma_{k}}=t_{\gamma_{k}}$ are tautologies, then we may replace the formula with $\top$, removing the quantifiers $\forall y_{0}\in T_{0}\dots\forall y_{m}\in T_{m}$, thus turning the negation of a special formula into an existential $\F(\B)^{*}$-formula as desired.

Now suppose that none of the equations are tautologies and consider $\forall y_{m}\in T_{m}\bigvee_{j}x_{\beta_{j}}=t_{\beta_{j}}\vee\bigvee_{k}y_{\gamma_{k}}=t_{\gamma_{k}}$.
For any valuation function $s:V\to\F(\B)$, consider $d_{s}\in\F(\B)$ chosen as follows.
Let $d_{s}\in T_{m}$ be so that $\Delta(d_{s})$ is strictly larger than $\Delta(s(x_{\beta_{j}}))$, $\Delta(s(y_{\gamma_{k}}))$ where $y_{\gamma_{k}}\neq y_{m}$, and also $\Delta(s(t_{\gamma_{k}}))$ where $y_{\gamma_{k}}=y_{m}$.
Since the total number of terms at each height is finite, such a choice of $d_{s}$ must exist.
Furthermore, for such a choice of $d_{s}$, note that $\hat{s}:V\to\F(\B)$ where $\hat{s}(z)=s(z)$ for all $z\neq y_{m}$ and $\hat{s}(y_{m})=d_{s}$ is such that the interpretation of $\bigvee_{j}x_{\beta_{j}}=t_{\beta_{j}}\vee\bigvee_{k}y_{\gamma_{k}}=t_{\gamma_{k}}$ evaluates to $\bot$ under $\hat{s}$.
That is, for any $s:V\to\F(\B)$, $\forall y_{m}\in T_{m}\bigvee_{j}x_{\beta_{j}}=t_{\beta_{j}}\vee\bigvee_{k}y_{\gamma_{k}}=t_{\gamma_{k}}$ also evaluates to $\bot$.
Then, we may remove the quantifier $\forall y_{m}\in T_{m}$ and all equations mentioning $y_{m}$ from the formula above.

Conversely, if $T_{m}$ is finite, then we split into two further sub-cases.
First, if there is a predicate $\is{f}(y_{m})$, then $T_{m}\cap\B=\emptyset$.
Additionally, by picking some $b_{0},\dots,b_{\arity(f)-1}\notin\B$, we can generate infinitely many distinct elements starting with $f$ and satisfying $\is{f}$.
Thus, if $T_{m}$ is to be finite, we necessarily have that $\neg\is{f}(y_{m})$ also occurs in the formula, or that $\F(\B)=\B$, resulting in $T_{m}=\emptyset$.

Second, if the only predicates mentioning $y_{m}$ are of the form $\neg\is{f}(y_{m})$, then for some other function symbol $g\neq f$, we can similarly generate infinitely many distinct elements satisfying $\neg\is{f}$ provided that $\F(\B)\setminus\B\neq\emptyset$.
That is, either $\neg\is{g}(y_{m})$ also occurs in the formula for all other function symbols $g$ in the language of $\B$, or $\F(\B)=\B$.
In the former, we evidently have that $T_{m}=\B$, and in the latter we similarly obtain that either $T_{m}=\B$ or $T_{m}=\emptyset$.

\smallskip

Therefore, each quantifier $\forall y_{i}\in T_{i}$ can either be removed, simultaneously removing all equations mentioning the removed bound variable; or the quantification can be replaced by a quantification over $\B$; or the formula can be replaced by $\top$.
\end{proof}

Following Remark \ref{rem:expblowup}, and a careful analysis of the proof, the resulting existential formula obtained after removing all universal quantifiers should have length $O\big(k^{|\varphi|}|\B|^{2|\varphi|}|\varphi|\big)$, where $k$ is the maximum arity of the function symbols.
Together with Remark \ref{rem:elength}, given $\varphi$ standard, the length of the standard formula $\psi\equiv\neg\exists x\varphi$ obtained from applying Lemmas \ref{lem:eformula} and \ref{lem:negstandard} is $O\big(k^{|\B|^{|\varphi|}|\varphi|}|\B|^{2|\B|^{|\varphi|}|\varphi|}|\B|^{|\varphi|}|\varphi|\big)$.
Thus, the quantifier elimination procedure is not bounded by any elementary recursive function.

\begin{definition}\label{def:closedstandard}
A standard formula is {\bf closed} if it contains no free variables.
That is, it must be constructed from special formulas and the logical connectives $\vee,\wedge$.
\end{definition}

It follows from Lemmas \ref{lem:eformula} and \ref{lem:negstandard} that any $\F(\B)^{*}$-formula with no free variables is equivalent to a closed standard formula.
Thus, to express $\F(\B)^{*}$-formulas as an equivalent quantifier-free formula, it suffices to show the following.

\begin{lemma}
Every special formula with no free variables is $\F(\B)^{*}$-equivalent to a quantifier-free formula.
\end{lemma}

\begin{proof}
From Definition \ref{def:standardformula}, the negation of a special formula with no free variables contains no non-equations.
Then by applying the process in Lemma \ref{lem:negstandard} to it, we may rewrite it as a quantifier-free formula with no variables.
\end{proof}

\begin{remark}\label{rem:extension}
Note that in the proofs of Lemmas \ref{lem:eformula} and \ref{lem:negstandard}, the finiteness of $\B$ is only used when expressing $x\in\B$ or $x\notin\B$ as a finite disjunction or conjunction respectively.
In view of this, we posit that by allowing $\in\B$ as a predicate to our language, our result can be easily extended to cover partial algebras $\B$ which are infinite and admit quantifier elimination.
We state this as a proposition without proof below.
\end{remark}

\begin{proposition}
If $\B$ admits quantifier elimination, then the expansion of $\F(\B)$ with $\is{f}$ predicates and $\in\B$ with the obvious semantics also admits quantifier elimination. $\hfill\qed$
\end{proposition}

\section{A non-initial algebra with polynomial time word problem}\label{sec:app}

Recall that an {\bf initial algebra} is an algebra isomorphic to $\F_{\Gamma}$ where $\Gamma$ is a finite set of equations (not necessarily of ground terms).

\begin{remark}
If $\Sigma$ has only a single unary function symbol, then all algebras of $\Sigma$ are initial.
To see this,  observe that the free term algebra with signature $\Sigma$ consists of distinct orbits generated by the single unary function symbol.
Any quotient could cause orbits `joining' up from some finite point, or orbits becoming eventually periodic.
\end{remark}

\begin{theorem}\label{thm:app}
For any signature $\Sigma$ which contains at least one constant symbol, and at least one binary functions symbol or two unary function symbols, there exists a non-initial algebra $\A$ over $\Sigma$ with a polynomial time word problem.
\end{theorem}

\begin{proof}
Let $\{\Gamma_{n}\}_{n\in\mathbb{N}}$ be an effective listing of all finite sets of equations over $\Sigma$.
We construct $\A=\F_{E}$ in stages, where $E=\bigcup_{s}E_{s}$ is a set of ground term equations over $\Sigma$.
By a careful definition of $E$, for each term $t\in\F$, the height remains the same as a term $t\in\F_{E}$.
\begin{description}
\item[Construction]
Let $E_{0}=\emptyset$.
At stage $s+1$, check if $\F_{E_{s}}\vDash\Gamma_{s}$.
If it does, then search for $p_{s},q_{s}$ such that $\Delta(p_{s})=\Delta(q_{s})>\max_{i<s}\{\Delta(\hat{p}_{i}),\Delta(\hat{q}_{i})\}$ and $\F_{E_{s}}\vDash p_{s}\neq q_{s}$.
For the first such pair $p_{s},q_{s}$ found, define $\hat{p}_{s},\hat{q}_{s}$ to be some {\bf padded terms} (to be defined based on the signature) of size\footnote{Recall that the size of $p$ is $|\Omega(p)|$ and is generally exponential in the height $\Delta(p)$.} larger than the total time the construction has run so far and define $E_{s+1}=E_{s}\cup\{\hat{p}_{s}=\hat{q}_{s}\}$. 
(Note that each stage will generally run for some non-elementary, in $s$, time, since the quantifier elimination procedure cannot be elementary recursive.)
Otherwise, search for $p_{s},q_{s}$ such that $\Delta(p_{s})\geq\Delta(q_{s})>\max_{i<s}\{\Delta(\hat{p}_{i}),\Delta(\hat{q}_{i})\}$, $\F_{\Gamma_{s}}\vDash p_{s}=q_{s}$, and $\F_{E_{s}}\vDash p_{s}\neq q_{s}$.
Then define $E_{s+1}=E_{s}$, $\hat{p}_{s}=p_{s}$ and $\hat{q}_{s}=q_{s}$.
\end{description}

Now we verify that the construction works.
First, we claim that $p_{s},q_{s}$ with the desired properties can always be found.
For the base case, note that $\F_{E_{0}}=\F$.
If $\F_{E_{0}}\vDash\Gamma_{0}$, then we must be able to find $p_{0},q_{0}$ for which $\Delta(p_{0})=\Delta(q_{0})$ and such that $\F_{E_{0}}\vDash p_{0}\neq q_{0}$ as desired.

If our signature $\Sigma$ contains at least one binary function symbol $f$, then we can define the padded term $\hat{p}_{0},\hat{q}_{0}$ respectively as $ff\dots fp_{0}\dots p_{0}$ and $ff\dots fq_{0}\dots q_{0}$, of the requisite size.
If our signature $\Sigma$ does not contain any binary function symbols, then define $\hat{p}_{0},\hat{q}_{0}$ respectively as $ff\dots fp_{0}$ and $gg\dots gq_{0}$ where $f,g$ are two distinct unary function symbols in $\Sigma$.
Observe that in either case, if $\F_{E_{0}}\vDash p_{0}\neq q_{0}$, then $\F_{E_{0}}\vDash\hat{p}_{0}\neq\hat{q}_{0}$.

Now, if $\F_{E_{0}}\not\vDash\Gamma_{0}$, then there must exist witnesses $p_{0},q_{0}$ such that $\F_{E_{0}}\vDash p_{0}\neq q_{0}$ but $\F_{\Gamma_{0}}\vDash p_{0}=q_{0}$.
Since equality in $\F_{\Gamma_{0}}$ is computably enumerable, and we know such witnesses exist, we can effectively find $p_{0},q_{0}$ with the desired properties.

Suppose that for each $i<s$, $p_{i},q_{i}$ are such that $\Delta(p_{i})\geq\Delta(q_{i})>\max_{j<i}\{\Delta(\hat{p}_{j}),\Delta(\hat{q}_{j})\}$, and if $\hat{p}_{i}=\hat{q}_{i}\in E_{s}$, then $\Delta(\hat{p}_{i})=\Delta(\hat{q}_{i})$.
If $\F_{E_{s}}\vDash\Gamma_{s}$, then we pick $p_{s},q_{s}$ as follows.
If $\Sigma$ contains a binary function symbol $f$, then pick $p_{s}=f\hat{p}_{s-1}f\hat{p}_{s-1}\hat{p}_{s-1}$ and $q_{s}=ff\hat{p}_{s-1}\hat{p}_{s-1}\hat{p}_{s-1}$.
By the assumption that $\Delta(\hat{p}_{i})=\Delta(\hat{q}_{i})$ if $\hat{p}_{i}=\hat{q}_{i}\in E_{s}$ for each $i<s$, $\F_{E_{s}}\not\vDash \hat{p}_{s-1}=f\hat{p}_{s-1}\hat{p}_{s-1}$.
Then, it follows that $\F_{E_{s}}\vDash p_{s}\neq q_{s}$, and thus $p_{s},q_{s}$ satisfy the desired properties.
Then, we may define $\hat{p}_{s}$ and $\hat{q}_{s}$ in the same way as before.
If $\Sigma$ contains two unary function symbols $f,g$, a similar argument suffices to verify that $p_{s}=f\hat{p}_{s-1}$ and $q_{s}=g\hat{p}_{s-1}$ also works.
Once again, we pad out $p_{s},q_{s}$ to obtain $\hat{p}_{s}=ff\dots fp_{s}$ and $\hat{q}_{s}=gg\dots gq_{s}$.
For the case where $\F_{E_{s}}\not\vDash\Gamma_{s}$, the same argument as in the base case together with the observation that there must exist infinitely many witnesses gives the desired conclusion.

Second, we verify that $\F_{E}$ where $E=\bigcup_{s}E_{s}$ has a polynomial time word problem.
Since every equation $p=q$ contained in $E$ is such that $\Delta(p)=\Delta(q)$, then for any ground terms $u,t$, if $\Delta(u)\neq\Delta(t)$, then $\F_{E}\vDash s\neq t$.
If on the other hand $\Delta(u)=\Delta(t)$, run the construction until a stage $s$ at which $|u|$ units of time have passed.
Recall that $s$ will be much smaller than $|u|$.
At such a stage, note that $|E_{s}|=\sum_{p=q\in E_{s}}(|p|+|q|)$ is bounded by $2|u|^{2}$.
Furthermore, any subsequent equation added to $E$ concerns only terms of strictly larger height than $u$ and so, the equality of $u,t$ cannot be affected by them.
That is, $\F_{E}\vDash u=t$ if and only if $\F_{E_{s}}\vDash u=t$.
Then, by using Kozen's algorithm, we obtain that checking if $\F_{E}\vDash u=t$ is polynomial time in $|E_{s}|,|u|,|t|$ which is polynomial in $|u|,|t|$.

Finally, we verify that $\F_{E}$ is not initial.
More specifically, that $\F_{E}\not\cong\F_{\Gamma_{s}}$ for all $s$.
Observe that $\F_{E}$ is always a homomorphic image of $\F_{E_{s}}$ for any finite $s$.
In addition, none of $\F_{E_{s}}$ is isomorphic to $\F_{E}$.
In this way, if $\Gamma_{s}$ is such that $\F_{E_{s}}\vDash\Gamma_{s}$, then $\F_{E}\not\cong\F_{\Gamma_{s}}$.
On the other hand, if $\F_{E_{s}}\not\vDash\Gamma_{s}$, then let $p_{s},q_{s}$ be the witnesses to  for which $\F_{\Gamma_{s}}\vDash p_{s}=q_{s}$ but $\F_{E_{s}}\vDash p_{s}\neq q_{s}$.
In this case, recall that we define $\hat{p}_{s}=p_{s}$ and $\hat{q}_{s}=q_{s}$.
Then, to see that $\F_{E}\vDash \hat{p}_{s}\neq \hat{q}_{s}$, note that any equation $\hat{p}_{t}=\hat{q}_{t}$ later added into $E$, is such that $\Delta(\hat{p}_{t})=\Delta(\hat{q}_{t})>\max\{\Delta(\hat{p}_{s}),\Delta(\hat{q}_{s})$, and so, cannot make $\F_{E}\vDash\hat{p}_{s}=\hat{q}_{s}$.
Thus, we obtain that for all $s$, $\F_{E}\not\cong\F_{\Gamma_{s}}$.
\end{proof}

\bibliographystyle{plain}
\bibliography{mybib}

@inproceedings{kozen77,
    author = {Dexter Campbell Kozen},
    title = {Complexity of finitely presented algebras},
    booktitle = {Proceedings of the ninth annual ACM symposium on Theory of computing},
    year = {1977},
    month = {May},
    pages = {164-177}
}

@inproceedings{sturm2002,
    author = {Sturm, Thomas and Weispfenning, Volker},
    year = {2002},
    month = {09},
    title = {Quantifier elimination in term algebras: The case of finite languages},
    booktitle = {Computer Algebra in Scientific Computing (CASC), TUM Muenchen},
    pages={285--300}
}

@article{mal1936,
  title={Axiomatizable classes of locally free algebras of various types},
  author={Mal'cev, Anatoly Ivanovich},
  journal={The metamathematics of algebraic systems. Collected papers: 1936-1967},
  pages={262--281},
  year={1971},
  publisher={North Holland Publishing Company}
}

@article{Rabin1969,
  author  = {Rabin, Michael O.},
  title   = {Decidability of Second-Order Theories and Automata on Infinite Trees},
  journal = {Transactions of the American Mathematical Society},
  volume  = {141},
  year    = {1969},
  pages   = {1--35},
  publisher = {American Mathematical Society}
}

@article{Belegradek1988,
  author  = {Oleg. V. Belegradek},
  title   = {Teoriya modelei lokal'no svobodnykh algebr},
  journal = {Trudy Instituta Matematiki Sibirskogo Otdeleniya AN SSSR},
  volume  = {8},
  pages   = {3--25},
  year    = {1988},
  note    = {In Russian}
}

@book{Hodges1993,
  author    = {Wilfrid Hodges},
  title     = {Model Theory},
  publisher = {Cambridge University Press},
  year      = {1993}
}

@article{compton1990,
  title={A uniform method for proving lower bounds on the computational complexity of logical theories},
  author={Compton, Kevin J and Henson, C Ward},
  journal={Annals of pure and applied logic},
  volume={48},
  number={1},
  pages={1--79},
  year={1990},
  publisher={Elsevier}
}

@article{Marongiu1993,
  author    = {Giampaolo Marongiu and Stefano Tulipani},
  title     = {Undecidable Fragments of Term Algebras with Subterm Relation},
  journal   = {Fundamenta Informaticae},
  volume    = {19},
  number    = {3-4},
  pages     = {371--382},
  year      = {1993},
  publisher = {IOS Press},
  doi       = {10.3233/FI-1993-193-408}
}

@article{Tulipani1994,
  author  = {Sauro Tulipani},
  title   = {Decidability of the Existential Theory of Infinite Terms with Subterm Relation},
  journal = {Information and Computation},
  volume  = {108},
  number  = {1},
  pages   = {1--33},
  year    = {1994},
  publisher = {Academic Press},
  doi       = {10.1006/inco.1994.1001}
}

@inproceedings{Zhang2004,
    author={Zhang, Ting and Sipma, Henny B. and Manna, Zohar},
    title={Term Algebras with Length Function and Bounded Quantifier Alternation},
    booktitle={Theorem Proving in Higher Order Logics},
    year={2004},
    publisher={Springer Berlin Heidelberg},
    pages={321--336}
}

@article{zhang2006,
  title={Decision procedures for term algebras with integer constraints},
  author={Zhang, Ting and Sipma, Henny B and Manna, Zohar},
  journal={Information and Computation},
  volume={204},
  number={10},
  pages={1526--1574},
  year={2006},
  publisher={Elsevier}
}

@inproceedings{korovin2000,
  title={A decision procedure for the existential theory of term algebras with the Knuth-Bendix ordering},
  author={Korovin, Konstantin and Voronkov, Andrei},
  booktitle={Proceedings Fifteenth Annual IEEE Symposium on Logic in Computer Science (Cat. No. 99CB36332)},
  pages={291--302},
  year={2000},
  organization={IEEE},
  doi={10.1109/LICS.2000.855777}
}

@article{Rybina2001,
  author = {Rybina, Tatiana and Voronkov, Andrei},
  year = {2001},
  month = {04},
  pages = {155-181},
  title = {A Decision Procedure for Term Algebras With Queues},
  volume = {2},
  journal = {ACM Transactions on Computational   Logic},
  doi = {10.1145/371316.371494}
}

@article{comon1993,
  title={Complete axiomatizations of some quotient term algebras},
  author={Comon, Hubert},
  journal={Theoretical Computer Science},
  volume={118},
  number={2},
  pages={167--191},
  year={1993},
  publisher={Elsevier}
}

@article{Bakhadyr2005,
  author = {Khoussainov, Bakhadyr and Rubin, Sasha},
  year = {2005},
  month = {08},
  pages = {292-308},
  title = {Decidability of Term Algebras Extending Partial Algebras},
  volume = {3634},
  journal = {Lecture Notes in Computer Science},
  doi = {10.1007/11538363_21}
}

@article{nov55,
    author = {Pyotr Sergeyevich Novikov},
    title = {On the algorithmic unsolvability of the word problem in group theory},
    journal = {Proceedings of the Steklov Institute of Mathematics},
    year = {1955},
    volume = {44},
    pages = {1-143}
}

@article{boone58,
    author = {William W. Boone},
    title = {The word problem},
    journal = {Annals of Mathematics},
    year = {1959},
    volume = {70},
    issue = {2},
    pages = {207-265},
    doi = {10.2307/1970103}
}

\end{document}